\documentclass[a4paper]{amsart}
\usepackage[latin1]{inputenc}
\usepackage{amsfonts}
\usepackage{amsmath,latexsym,amssymb,amsfonts, amsthm}
\usepackage{esint}
\usepackage{cite}
\usepackage{graphicx}
\usepackage{amscd}
\usepackage{color}
\usepackage{bm}             
\usepackage{enumerate}
\usepackage[dvips]{epsfig}
\usepackage{psfrag}
\usepackage{epsfig}

\usepackage[showonlyrefs]{mathtools}
\mathtoolsset{showonlyrefs=true}

\usepackage[
  hmarginratio={1:1},     % equal left and right margins
  vmarginratio={1:1},     % equal top and bottom margins=
  textwidth=15cm,        % new text width
  textheight=21cm,
  heightrounded,          % always useful
]{geometry}

\usepackage{graphicx,color}
\usepackage[colorlinks]{hyperref}
\hypersetup{linkcolor=blue,citecolor=blue,filecolor=black,urlcolor=blue}

\newtheorem{theorem}{Theorem}
\newtheorem{corollary}[theorem]{Corollary}
\newtheorem{proposition}[theorem]{Proposition}
\newtheorem{lemma}[theorem]{Lemma}
\theoremstyle{definition}
\newtheorem{remark}{Remark}

\newtheorem{definition}{Definition}

\numberwithin{theorem}{section}

\numberwithin{remark}{section}

\numberwithin{equation}{section}

\newcommand{\R}{\mathbb{R}}
\newcommand{\N}{\mathbb{N}}

\newcommand{\pa}{\partial}

\newcommand{\C}{{\mathbb C}}

\newcommand{\cD}{{\mathcal D}}

\newcommand{\cP}{{\mathcal P}}

\newcommand{\cS}{{\mathcal S}}

\newcommand{\dist}{{\rm dist}}

\newcommand{\weak}{\rightharpoonup}

\newcommand{\eps}{\varepsilon}

\DeclareMathOperator{\rad}{rad}

\title[Normalized ground states for the combined Sobolev critical NLSE]{Normalized ground states for the NLS equation with combined nonlinearities: the Sobolev critical case}

\author[N. Soave]{Nicola Soave}\thanks{}
\address{Nicola Soave \newline \indent
Dipartimento di Matematica,  Politecnico di Milano,  \newline \indent
Via Edoardo Bonardi 9, 20133 Milano, Italy}
\email{nicola.soave@gmail.com; nicola.soave@polimi.it}

\thanks{The author is partially supported by the ERC Advanced Grant 2013 n. 339958 ``Complex Patterns for Strongly Interacting Dynamical Systems - COMPAT'', and by the PRIN-2015KB9WPT\texttt{\char`_}010 Grant: ``Variational methods, with applications to problems in mathematical physics and geometry". The author is a member of the INDAM-GNAMPA group.}

\begin{document}

\begin{abstract}
We study existence and properties of ground states for the nonlinear Schr\"odinger equation with combined power nonlinearities
\[
-\Delta u= \lambda u + \mu |u|^{q-2} u + |u|^{2^*-2} u \qquad \text{in $\R^N$, $N \ge 3$,}
\]
having prescribed mass
\[
\int_{\R^N} |u|^2 = a^2,
\]
in the \emph{Sobolev critical case}. For a $L^2$-subcritical, $L^2$-critical, of $L^2$-supercritical perturbation $\mu |u|^{q-2} u$ we prove several existence/non-existence and stability/instability results. 

This study can be considered as a counterpart of the Brezis-Nirenberg problem in the context of normalized solutions, and seems to be the first contribution regarding existence of normalized ground states for the Sobolev critical NLSE in the whole space $\R^N$.
\end{abstract}

\maketitle

\section{Introduction}

In this paper we study existence and properties of ground states with prescribed mass for the nonlinear Schr\"odinger equation with combined power nonlinearities
\begin{equation}\label{com nls}
i \psi_t + \Delta \psi + \mu |\psi|^{q-2} \psi + |\psi|^{p-2} \psi = 0 \qquad \text{in $\R^N$},
\end{equation}
in the case when $N \ge 3$ and 
\[
p=2^* = \frac{2N}{N-2}
\]
is the critical exponent for the Sobolev embedding $H^1(\R^N) \hookrightarrow L^p(\R^N)$. The NLS equation with combined nonlinearities attracted much attention in the last decade, starting from the fundamental contribution by T. Tao, M. Visan and X. Zhang \cite{TaoVisZha}. According to \cite[Section 3]{TaoVisZha} (based on  \cite{Caz, CazWei}), the Cauchy problem for \eqref{com nls} is locally well posed, and the unique strong solution in $C((T_{\min}, T_{\max}), H^1(\R^N))$ has conservation of \emph{energy} 
\begin{equation}\label{def E}
E_\mu: H^1(\R^N, \C) \to \R, \quad E_\mu(u) = \int_{\R^N} \left( \frac{1}{2} |\nabla u|^2 -\frac{1}{2^*} |u|^{2^*} - \frac{\mu}{q} |u|^q \right)
\end{equation}
and of \emph{mass} 
\[
|u|_2^2 := \int_{\R^N} |u|^2
\]
(here an in the rest of the paper, $|\cdot|_2$ denotes the standard norm in $L^2(\R^N,\C)$). Global well-posedness, scattering, the occurrence of blow-up and more in general dynamical properties has been studied in \cite{TaoVisZha} and many other papers. In particular, the Sobolev critical case has been considered in \cite{AkaIbrKikNaw, AkaIbrKikNaw2, AkaIbrKikNaw3, ChMiZh, KiOhPoVi, MiXuZh, MiaZhaZhe} (see also the references therein). In this paper we continue the study initiated in \cite{So} concerning existence and properties of ground states with prescribed mass for the NLS equation with combined nonlinearities; while in \cite{So} we addressed the Sobolev subcritical case, here we focus on the Sobolev critical one. 

To find stationary states, one makes the ansatz $\psi(t,x) = e^{-i \lambda t} u(x)$, where $\lambda \in \R$ and $u: \R^N \to \C$ is a time-independent function. This ansatz yields
\begin{equation}\label{stat com}
-\Delta u  = \lambda u + |u|^{2^*-2} u + \mu |u|^{q-2} u \qquad \text{in $\R^N$}.
\end{equation}
A possible choice is then to fix $\lambda \in \R$, and to search for solutions to \eqref{stat com} as critical points of the \emph{action functional}
\[
\mathcal{A}(u):= \int_{\R^N} \left(\frac{1}{2} |\nabla u|^2 -\frac{\lambda}2 |u|^2 - \frac{\mu}{q} |u|^q - \frac{1}{2^*} |u|^{2^*}\right);
\]
in this case particular attention is devoted to \emph{least action solutions}, namely solutions minimizing $\mathcal{A}$ among all non-trivial solutions. In this direction, we refer to \cite{AlSoMo, FerGaz}, where the existence of positive decaying real solutions for elliptic equations in $\R^N$ is addressed in a very general setting; to \cite{AkaIbrKikNaw, AkaIbrKikNaw2}, which concern the case when $q>2+4/N$ and $\mu>0$; to \cite{ChMiZh, MiXuZh}, where the variational framework for the fixed $\lambda$ problem is analyzed in the $L^2$-critical $q=2+4/N$ case with $\mu<0$; to \cite{KiOhPoVi}, where the authors study the focusing-cubic defocusing-quintic NLS in $\R^3$.

Alternatively, one can search for solutions to \eqref{stat com} having prescribed mass, and in this case $\lambda \in \R$ is part of the unknown. This approach seems to be particularly meaningful from the physical point of view, and often offers a good insight of the dynamical properties of the stationary solutions for \eqref{com nls}, such as stability or instability \cite{BerCaz,CazLio}. Here we focus on this second approach, which was, up to now, unexplored. 

The existence of normalized stationary states can be formulated as follows: given $a>0$, $\mu \in \R$, and $2<q<2^*$, we aim to find $(\lambda,u) \in \R \times H^1(\R^N, \C)$ solving \eqref{stat com} together with the normalization condition 
\begin{equation}\label{norm}
|u|_2^2 = \int_{\R^N} |u|^2 = a^2.
\end{equation} 
Solutions can be obtained as critical points of the energy functional $E_\mu$ (defined in \eqref{def E}) under the constraint
\[
u \in S_a:= \left\{u \in H^1(\R^N,\C): \int_{\R^N} |u|^2 = a^2 \right\}.
\]
It is standard that $E_\mu$ is of class $C^1$ in $H^1(\R^N,\C)$, and any critical point $u$ of $E_\mu|_{S_a}$ corresponds to a solution to \eqref{stat com} satisfying \eqref{norm}, with the parameter $\lambda \in \R$ appearing as Lagrange multiplier. We will be particularly interested in ground state solutions, defined in the following way:

\begin{definition}
We write that $\tilde u$ is a \emph{ground state} of \eqref{stat com} on $S_a$ if it is a solution to \eqref{stat com} having minimal energy among all the solutions which belongs to $S_a$: 
\[
d E_\mu|_{S_a}(\tilde u) = 0 \quad \text{and} \quad E_\mu(\tilde u) = \inf\{E_\mu(u): \  d E_\mu|_{S_a}(u) = 0, \quad \text{and} \quad u \in S_a\}.
\]
The set of the ground states will be denoted by $Z_{a,\mu}$.
\end{definition}
This definition seems particularly suited in our context, since $E_\mu$ is unbounded from below on $S_a$, and hence global minima do not exist. We also recall the notion of stability and instability we will be interested in:

\begin{definition}
$Z_{a,\mu}$ is \emph{orbitally stable} if for every $\eps>0$ there exists $\delta>0$ such that, for any $\psi_0 \in H$ with $\inf_{v \in Z_{a,\mu}} \|\psi_0 - v\|_H < \delta$, we have
\[
\inf_{v \in Z_{a,\mu}} \|\psi(t,\cdot) - v\|_H < \eps \qquad \forall t>0,
\]
where $\psi(t, \cdot)$ denotes the solution to \eqref{com nls} with initial datum $\psi_0$. \\
A standing wave $e^{i \lambda t} u$ is \emph{strongly unstable} if for every $\eps>0$ there exists $\psi_0 \in H^1(\R^N,\C)$ such that $\|u-\psi_0\|_H <\eps$, and $\psi(t, \cdot)$ blows-up in finite time.
\end{definition}
We observe that the definition of stability implicitly requires that \eqref{com nls} has a unique global solution, at least for initial data $\psi_0$ sufficiently close to $Z_{a,\mu}$. 

\medskip

The search for normalized ground states, or more in general normalized solutions, for nonlinear Schr\"odinger equations such as
\[
-\Delta u = \lambda u + \mu |u|^{q-2} u + |u|^{p-2} u \quad \text{in $\R^N$}, \qquad \int_{\R^N} |u|^2 = a^2
\]
is a challenging and interesting problem (apart from the homogeneous case $p=q$, which can be reduced to the fixed-$\lambda$-problem by scaling). The presence of the $L^2$-constraint makes several methods developed to deal with unconstrained variational problems unavailable, and new phenomena arise. One for all, a new critical exponent appears, the \emph{$L^2$-critical exponent} 
\[
\bar p: = 2+ 4/N.
\] 
This is the threshold exponent for many dynamical properties such as global existence vs. blow-up, and the stability or instability of ground states. From the variational point of view, if the problem is purely $L^2$-subcritical, i.e. $2<q<p<\bar p$, then $E_\mu$ is bounded from below on $S_a$. Thus, for every $a, \mu >0$ a ground state can be found as global minimizers of $E_\mu|_{S_a}$, see \cite{Stu1} or \cite{Lions2,Shi}. Moreover, the set of ground states is orbitally stable \cite{CazLio, Shi}. In the purely $L^2$-supercritical case, i.e. $\bar p<q<p<2^*$, on the contrary, $E_\mu|_{S_a}$ is unbounded from below; however, exploiting the mountain pass lemma and a smart compactness argument, L. Jeanjean \cite{Jea} could show that a normalized ground state does exist for every $a, \mu>0$ also in this case. The associated standing wave is strongly unstable \cite{BerCaz, LeCoz}, due to the supercritical character of the equation. We point out that, in \cite{Jea, LeCoz, Lions2, Shi, Stu1}, more general nonlinearities are considered. 

In \cite{So} we studied what happens when the combined power nonlinearities in \eqref{com nls} are of mixed type, that is
\[
2<q \le 2 +\frac{4}{N} \le p<2^*, \quad \text{with $p \neq q$ and $\mu \in \R$}.
\]
We saw that the interplay between subcritical, critical and supercritical nonlinearities strongly affects the geometry of the functional and the existence and properties of ground states. Here we continue the analysis of mixed problems, focusing on the choice $p=2^*$, and allowing $q$ to be $L^2$-subcritical, $L^2$-critical, or $L^2$-supercritical. The whole study can be considered as a counterpart of the Brezis-Nirenberg problem in the context of normalized solutions: we have a homogeneous problem for which the structure of the ground states is known, and we analyze how the introduction of a lower order term modifies this structure. In this perspective, we think that it is natural to treat the coefficient $\mu$ in front of $|u|^{q-2} u$ as a parameter, fixing the coefficient of $|u|^{2^*-2} u$ in \eqref{stat com} to be $1$. Notice however that, by scaling, it is possible to reverse this choice when $\mu>0$. Moreover, since the coefficient of the $|u|^{2^*-2}u$ is positive, we point out that we always consider a focusing ``leading" nonlinearity. 

\medskip 

Since the exponent $2^*$ is $L^2$-supercritical, the functional $E_\mu$ is always unbounded from below on $S_a$. For quite a long time the paper \cite{Jea} was the only one dealing with existence of normalized solutions in cases when the energy is unbounded from below on the $L^2$-constraint. One of the main difficulties that one has to face in such context is the analysis of the convergence of constrained Palais-Smale sequences: indeed, even in a Sobolev subcritical framework, the mere boundedness of a Palais-Smale sequence is not guaranteed in general; sequences of approximated Lagrange multipliers have to be controlled (since $\lambda$ is not prescribed); and moreover, weak limits of Palais-Smale sequences could leave the constraint, since the embeddings $H^1(\R^N) \hookrightarrow L^2(\R^N)$ and also $H^1_{\rad}(\R^N) \hookrightarrow L^2(\R^N)$ are not compact. In \cite{Jea}, L. Jeanjean could overcome these obstructions showing that the mountain pass geometry of $E_\mu|_{S_a}$ allows to construct a Palais-Smale sequence of functions satisfying the Pohozaev identity. This gives boundedness, which is the first step in proving strong $H^1$-convergence.  

More recently, this kind of idea has been exploited and further developed in other contexts, and more in general the search for normalized solutions in cases when the energy is unbounded from below on the $L^2$-constraint attracted much attention: we refer to \cite{AcWe, BaDeV, BeJe, BeJeLu, BoCaGoJe, BuEsSe, JeLuWa} for normalized solutions to scalar equations in the whole space $\R^N$, to \cite{BaSo, BaSo2, BaJe, BaJeSo, GoJe} for normalized solutions to systems in $\R^N$, and to \cite{FibMer, NoTaVe1, NoTaVe2, NoTaVe3, PiVe} for normalized solutions to equations or systems in bounded domains. We mention that in all the aforementioned references, with the exception of \cite{NoTaVe3}, Sobolev subcritical problems are considered. In particular, the present paper seems to be the first result for normalized solutions of a Sobolev critical problem in the whole space $\R^N$. As naturally expected, the presence of the Sobolev critical term in \eqref{stat com} further complicates the study of the convergence of Palais-Smale sequences. One of the most relevant aspects of our study consists in showing that, suitably combining some of the main ideas from \cite{BreNir} and \cite{Jea}, compactness can be restored also in the present setting.

\subsection{Main results} Let 
\begin{equation}\label{def gamma_p}
\gamma_p = \frac{N(p-2)}{2p}, \qquad \forall p \in (2,2^*].
\end{equation}
We summarize our main existence result in the following statement:

%We will consider separately the $L^2$-subcritical, critical, and supercritical

%The following results concern, respectively, the cases when the lower order term $\mu |u|^{q-2} u$ is $L^2$-subcritical, $L^2$-critical, and $L^2$-supercritical.

\begin{theorem}\label{thm: main}
Let $N \ge 3$, $2<q<2^*$, and let $a,\mu>0$. There exists a constant $\alpha=\alpha(N,q)>0$ such that, if 
\begin{equation}\label{hp}
\mu a^{(1-\gamma_q)q} < \alpha,
\end{equation}
then $E_\mu|_{S_a}$ has a ground state $\tilde u$ with the following properties: $\tilde u$ is a real valued, positive, radially symmetric function, and solves \eqref{stat com} for some $\tilde \lambda<0$. Moreover, denoting by $m(a,\mu) = E_\mu(\tilde u)$, we have that:
\begin{itemize}
\item[1)] If $2<q<2+4/N$, then $m(a,\mu)<0$, and $\tilde u$ an interior local minimizer of $E_\mu$ on the set
\[
A_k:= \left\{ u \in S_a: \|\nabla u\|_{L^2(\R^N)}< k \right\},
\]
for a suitable $k>0$ small enough. Any other ground state of $E_\mu$ on $S_a$ is a local minimizer of $E_\mu$ on $A_k$. Moreover, $\tilde u$ is radially decreasing.
\item[2)] If $2+4/N \le q<2^*$, then $0<m(a,\mu)< \cS^\frac{N}2/N$, and $\tilde u$ is a critical point of mountain pass type.
\end{itemize}
\end{theorem}

Here and in the rest of the paper, $\cS$ denotes the best constant for the Sobolev inequality, see \eqref{Sob ineq}. As in \cite{So}, condition \eqref{hp} is not a perturbative assumption, and we have explicit estimates for $\alpha$ in terms of Gagliardo-Nirenberg and Sobolev constants according to the fact that $q$ is $L^2$-subcritical, $L^2$-critical, or $L^2$-supercritical (see \eqref{alpha sub}, \eqref{alpha cr}, and \eqref{alpha sup} respectively). For $2+4/N < q<2^*$, it is remarkable that we can prove that $\alpha(N,q)=+\infty$ for $N=3,4$, so that any $a,\mu>0$ are admissible (see \eqref{alpha sup}). In dimension $N \ge 5$, instead, the threshold $\alpha(N,q)$ is finite. The difference between the dimensions $N =3,4$ and $N \ge 5$ reflects the different integrability properties of the extremal functions for the Sobolev inequalities, and is already present in the study of the homogeneous problem. We refer to Proposition \ref{prop: hom} and Remark \ref{rem: on alpha} for more details.

Even though we decided to present the result in a unified form, the three cases will be treated separately, as the geometry of $E_\mu|_{S_a}$ changes according the behavior of $q$. As a matter of fact, the study of the geometry of $E_\mu|_{S_a}$ presents several analogies with the case where $2^*$ is replaced by an exponent $p \in (2+4/N,2^*)$, studied in \cite{So} (see Theorems 1.3 and 1.6 therein). On the contrary, the analysis of the convergence of Palais-Smale sequences has to be treated in a very different way. 

\begin{remark}
In the $L^2$-subcritical case $q <2+4/N$, since $E_\mu|_{S_a}$ is unbounded from below, it could be natural to expect that there exists a second positive real valued and radial critical point on $S_a$ (as in the analogue subcritical counterpart, cf. with \cite[Theorem 1.3]{So}). We can indeed prove the existence of a Palais-Smale sequence for $E_\mu|_{S_a}$ at a mountain pass level $\sigma(a,\mu) > m(a,\mu)$, but the convergence of such sequence is a very delicate problem, which at the moment we could not solve. 
\end{remark}

\begin{remark}
Assumption \eqref{hp} plays different roles in the two cases $2<q\le 2+4/N$, and $2+4/N < q<2^*$. In the latter case $2+4/N < q<2^*$, assumption \eqref{hp} is used in order to ensure that the ground state level $m(a,\mu)$ is less than $\cS^\frac{N}2/N$, which is an essential ingredient in our compactness argument. If instead $2<q\le 2+4/N$, then \eqref{hp} enters also in the study of the geometry of the constrained functional $E_\mu|_{S_a}$. \end{remark}

%
%
%\begin{theorem}
% critical point $\tilde u$ at negative level $m(a,\mu)<0$, which is an interior local minimizer  on the set
%\[
%A_k:= \left\{ u \in S_a: \int_{\R^N} u^2 =a^2 \right\},
%\]
%for a suitable $k>0$ small enough. Moreover, $\tilde u$ is a ground state of $E|_{S}$, and any other ground state is a local minimizer of $E$ on $A_k$. Finally, $\tilde u$ is a real valued radially symmetric and decreasing function, and solves \eqref{stat com} for some $\tilde \lambda<0$.
%\end{theorem}
%
%\begin{theorem}\label{thm: cr}
%Let $N \ge 3$, $q=\bar p<p<2^*$, $a,\mu>0$. If
%\begin{equation}\label{hp cr pos}
%\mu a^{\frac{4}{N}} < \frac{\bar p}{2 C_{N,\bar p}^{\bar p}},
%\end{equation}
%then $E_\mu|_{S_a}$ has a critical point $\tilde u$ of mountain pass type at positive level $0<m(a,\mu)< \cS^\frac{N}2/N$, with the following properties: $\tilde u$ is a real-valued positive function in $\R^N$, is radially symmetric, solves \eqref{stat com} for some $\tilde \lambda<0$, and is a ground state of \eqref{stat com} on $S_a$.
%\end{theorem}
%
%\begin{theorem}\label{thm: supcr}
%Let $N \ge 3$, $\bar p<q<2^*$, $a,\mu>0$. There exists $\beta=\beta(N,q)>0$ such that, if 
%\begin{equation}\label{hp cr sup}
%\mu a^{(1-\gamma_q)q} < \beta,
%\end{equation}
%then $E_\mu|_{S_a}$ has a critical point $\tilde u$ of mountain pass type at positive level $0<m(a,\mu)< \cS^\frac{N}2/N$, with the following properties: $\tilde u$ is a real-valued positive function in $\R^N$, is radially symmetric, solves \eqref{stat com} for some $\tilde \lambda<0$, and is a ground state of \eqref{stat com} on $S_a$.
%\end{theorem}

It is an open problem to understand what happen if $\mu>0$ and $\mu a^{(1-\gamma_q)q}$ is large. In this case, we believe that ground states solutions do not exist in general. It is also natural to investigate what happen when $\mu<0$, that is, the ``perturbation term" is of defocusing type. To discuss our result under this assumption, we observe that in Theorem \ref{thm: main} we always found a ground state $\tilde u$ with the following properties:
\begin{itemize}
\item $\tilde u$ is a positive radially symmetric real-valued solution to \eqref{stat com}, for some $\lambda<0$;
\item $\tilde u$ is exponentially decaying at inifinity (this follows in a standard way as in \cite{BerLio}, once that we prove that $\tilde u$ is radial and that $\lambda<0$);
\item the energy level of $\tilde u$ is $E_\mu(u) < \cS^{\frac{N}2}/N$.
\end{itemize}
In particular, the first two properties are very natural in this kind of problems, and are satisfied by ground states of many NLS-type equations, see e.g. \cite{BaSo, BaDeV, Jea, So}. The third property is crucial in our compactness argument, and assumptions of this type are natural in Sobolev critical framework (cf. for instance with \cite{BreNir}). Therefore, it seems reasonable that, if a ground state do exist also for $\mu<0$, it enjoys at least some of these properties. With this in mind, we can prove a non-existence result:

\begin{theorem}\label{thm: non-ex}
Let $N \ge 3$, $a>0$, $\mu<0$, and let $2<q<2^*$. 
\begin{itemize}
\item[1)] For every $N \ge 3$, if $u$ is a critical point for $E_\mu|_{S_a}$ (not necessarily positive, or even real-valued) then the associated Lagrange multiplier $\lambda$ is positive, and $E_\mu(u) > \cS^{\frac{N}2}/N$.
\item[2)] If $N=3,4$, then the problem
\begin{equation}\label{non ex}
-\Delta u= \lambda u + \mu u^{q-1}  + u^{2^*-1}, \qquad u>0  \qquad \text{in $\R^N$} 
\end{equation} 
has no solution $u \in H^1(\R^N)$, for any $\lambda >0$ and $\mu<0$.
\item[3)] If $N \ge 5$, problem \eqref{non ex} has no solution $u \in H^1(\R^N)$ satisfying the additional assumption that $u \in L^p(\R^N)$ for some $p \in (0,N/(N-2)]$.
\end{itemize}
\end{theorem}

This means that, if a ground state does exist for $\mu<0$, it has completely different properties with respect to the case $\mu >0$: in particular, in dimension $N=3,4$, the theorem establishes that, if a critical point of $E_\mu|_{S_a}$ exists, then it cannot be real valued and positive. The same holds in higher dimension, provided that we also know that $u$ decays sufficiently fast at infinity (for instance, real valued positive exponentially decaying ground states cannot exist, in any dimension). We stress that this phenomenon is purely Sobolev-critical, since in the subcritical case we may have existence of real valued positive exponentially decaying ground states for $\mu<0$, see \cite[Theorem 1.9]{So}.

\medskip

We now turn to the study of the properties of the ground state solutions found in Theorem \ref{thm: main}. Recall that $Z_{a,\mu}$ denotes the set of ground states of $E_\mu$ on $S_a$. 

\begin{theorem}\label{thm: inst}
Under the assumption of Theorem \ref{thm: main}, it results that
\begin{equation}\label{struct Z}
Z_{a,\mu}:= \left\{ e^{i \theta} |u| \text{ for some $\theta \in \R$ and $|u|>0$ in $\R^N$}\right\}.
\end{equation}
Moreover, if $2+4/N \le q < 2^*$, and $u \in Z_{a,\mu}$, then the associated Lagrange multiplier is $\lambda<0$, and the associated standing wave $e^{-i \lambda t} u(x)$ is strongly unstable.
\end{theorem}

If $2<q<2+4/N$, ground states locally minimize the energy, and then on the contrary it seems natural that $Z_{a,\mu}$ is orbitally stable (as in the corresponding Sobolev subcritical case, see \cite[Theorem 1.4]{So}). In trying to prove such result, following the classical Cazenave-Lions argument \cite{CazLio}, two ingredients are essentials: the relative compactness of all the minimizing sequences for $E_\mu$ restricted on $A_k$, up to translations, and the global existence of solutions of the time dependent equation \eqref{com nls} for initial data close to $Z_{a,\mu}$. For the relative compactness, in \cite{So} we made use of the concentration compactness principle, but the proof in \cite{So} does not work in the Sobolev critical setting (since \cite[Lemma I.1]{Lions2} regards only subcritical exponents). Moreover, the global existence is affected by the presence of the critical exponent $2^*$ in the following way: in Sobolev subcritical cases, it is well known \cite{Caz} that if the maximal positive existence time $T_{\max}>0$ is finite, then necessarily
$\|\nabla \psi(t)\|_{L^2(\R^N)} \to +\infty$ as $t \to T_{\max}^-$ (and an analogue alternative holds for negative times). Thus, uniform a priori estimates on $\|\nabla \psi(t)\|_{L^2(\R^N)}$ yield global existence, and we used this strategy in the proof of \cite[Theorem 1.4]{So}. On the contrary, up to our knowledge, in the Sobolev critical case it is unknown whether the previous blow-up alternative holds, and hence we cannot deduce global existence from a priori bounds on $\|\nabla \psi(t)\|_{L^2(\R^N)}$ (see \cite[Theorem 4.5.1]{Caz} or \cite[Proposition 3.2]{TaoVisZha} for more details). As a consequence, the stability of the ground state in case of a subcritical perturbation remains an open problem.
%\begin{theorem}\label{thm: st}
%Let $N \ge 3$, $2<q<2+4/N$, and $a>0$. There exists $\tilde \mu>0$ depending on $a$, $q$ and on $N$ such that, if $0< \mu < \tilde \mu$, and if solutions to \eqref{com nls} with initial data sufficiently close to $Z_{a,\mu}$ are globally defined in time, then $Z_{a,\mu}$ is orbitally stable.
%\end{theorem}

We focus now on the behavior of the ground states found in Theorem \ref{thm: main} as $\mu \to 0^+$.

\begin{theorem}\label{thm: mu to 0}
Let $N \ge 3$ and $2<q<2^*$. For a fixed $a>0$, let $\mu>0$ be such that \eqref{hp} holds, and let $\tilde u_\mu$ be the corresponding positive radial ground state, with energy level $m(a,\mu)$.
\begin{itemize}
\item[1)] If $2<q<2+4/N$, then $m(a,\mu) \to 0$, and $\|\nabla \tilde u_\mu\|_{L^2(\R^N)}^2 \to 0$ as $\mu \to 0^+$. 
%Thus, $\tilde \mu \to 0$ strongly in $\cD^{1,2}(\R^N)$ and weakly in $H^1(\R^N)$. 
\item[2)] If $2+4/N \le q <2^*$, then $m(a,\mu) \to \cS^\frac{N}2/N$, and $\|\nabla \tilde u_\mu\|_{L^2(\R^N)}^2 \to \cS^{\frac{N}2}$ as $\mu \to 0^+$. 
%Moreover, if $N=3,4$ we have that $\tilde u_\mu \weak 0$ in $H^1(\R^N)$, while if $N \ge 5$ the following alternative holds: either $\tilde u_\mu \weak 0$ in $H^1(\R^N)$, or $\tilde u_\mu \to U_a$ strongly in $\cD^{1,2}(\R^N)$ and weakly in $H^1(\R^N)$, where $U_a$ is an extremal function for the Sobolev inequality with $L^2$-norm equal to $a$.
\end{itemize}
\end{theorem}
The difference with respect to the $L^2$-subcritical and $L^2$-critical/supercritical cases reflects the different variational characterization of $\tilde u_\mu$. Notice that, while in the former case we have a precise description of the asymptotic behavior of $\tilde u_\mu$ ($\tilde u_\mu$ tends to $0$ strongly in $\cD^{1,2}$ and weakly, but not strongly, in $H^1$), in the second case we can only describe the asymptotic behavior of $\|\nabla \tilde u_\mu\|_{L^2(\R^N)}$. The study of the convergence of $\tilde u_\mu$ represents a delicate problem, and we refer to Remark \ref{rem: on gs to 0} for a more detailed discussion.

\medskip

As in \cite{So}, a special role in the proof of Theorem \ref{thm: main} is played by the \emph{Pohozaev manifold}
\begin{equation}\label{def P}
\cP_{a,\mu} = \left\{u \in S_{a}: P_\mu(u) = 0\right\},
\end{equation}
where
\begin{equation}\label{Poh}
P_\mu(u) := \int_{\R^N} |\nabla u|^2 -  \int_{\R^N} |u|^{2^*} - \mu \gamma_q \int_{\R^N} |u|^q,
\end{equation}
and we recall that $\gamma_q$ was defined in \eqref{def gamma_p}. It is well known that any critical point of $E_\mu|_{S_a}$ stays in $\cP_{a,\mu}$, as a consequence of the Pohozaev identity (see \cite[Lemma 2.7]{Jea}). 
%
%Moreover, $\cP_{a,\mu}$ is \emph{a natural constraint}:
%
%\begin{proposition}\label{prop: natural}
%Suppose that the assumptions of Theorem \ref{thm: main} hold. Then $\cP_{a,\mu}$ is a smooth manifold of codimension $1$ in $S_a$, and if $u \in \cP_{a,\mu}$ is a critical point for $E_\mu|_{\cP_{a,\mu}}$, then $u$ is a critical point for $E_\mu|_{S_a}$.
%\end{proposition}
%
The properties of $\cP_{a,\mu}$ are intimately related to the minimax structure of $E_\mu|_{S_a}$, and in particular to the behavior of $E_\mu$ with respect to dilations preserving the $L^2$-norm (see Section \ref{sec: pre}).
To be more precise, for $u \in S_a$ and $s \in \R$, let
\begin{equation}\label{def star}
(s \star u)(x) := e^{\frac{N}{2}s} u(e^s x), \qquad \text{for a.e. $x \in \R^N$}.
\end{equation}
It results that $s \star u \in S_a$, and hence it is natural to study the \emph{fiber maps}
\begin{equation}\label{def psi}
\Psi^\mu_{u}(s) := E_\mu(s \star u) = \frac{e^{2s}}{2} \int_{\R^N} |\nabla u|^2 - \frac{e^{2^* s}}{2^*} \int_{\R^N} |u|^{2^*} -  \mu \frac{e^{q \gamma_q  s}}{q} \int_{\R^N} |u|^q.
\end{equation}
We will see that monotonicity and convexity properties of $\Psi^\mu_u$ strongly affect the structure of $\cP_{a,\mu}$, and also have a strong impact on properties of the the time-dependent equation \eqref{com nls}. Indeed, it is not difficult to check that $(\Psi_u^\mu)'(s) = P_\mu(s \star u)$, so that $s$ is a critical point of $\Psi_u^\mu$ if and only if $s \star u \in \cP_{a,\mu}$, and in particular $u \in \cP_{a,\mu}$ if and only if $0$ is a critical point of $\Psi_u^\mu$. In this spirit, we consider the decomposition of $\cP$ into the disjoint union $\cP_{a,\mu} = \cP_+^{a,\mu} \cup \cP_0^{a,\mu} \cup \cP_-^{a,\mu}$, where
\begin{equation}\label{split P}
\begin{split}
&\mathcal{P}_+^{a,\mu}  := \left\{ u \in \cP_{a,\mu}: 2|\nabla u|_2^2 > \mu q \gamma_q^2 |u|_q^q + 2^*  |u|_{2^*}^{2^*} \right\} =\left\{ u \in S_a: (\Psi_u^\mu)'(0) = 0, \ (\Psi_u^\mu)''(0) >0\right\}
  \\
&\mathcal{P}_-^{a,\mu}  := \left\{ u \in \cP_{a,\mu}: 2|\nabla u|_2^2 < \mu q \gamma_q^2 |u|_q^q + 2^* |u|_{2^*}^{2^*} \right\} =\left\{ u \in S_a: (\Psi_u^\mu)'(0) = 0, \ (\Psi_u^\mu)''(0) <0\right\}
 \\
&\mathcal{P}_0^{a,\mu}  := \left\{ u \in \cP_{a,\mu}: 2|\nabla u|_2^2 = \mu q \gamma_q^2 |u|_q^q + 2^* |u|_{2^*}^{2^*} \right\} =\left\{ u \in S_a: (\Psi_u^\mu)'(0) = 0, \ (\Psi_u^\mu)''(0) =0\right\}.
\end{split}
\end{equation}
Then we have:

\begin{proposition}\label{prop: struct P}
Suppose that the assumptions of Theorem \ref{thm: main} are satisfied. Then $\cP_{a,\mu}$ is a smooth manifold of codimension $1$ in $S_a$, and is a natural constraint, in the sense that if $u \in \cP_{a,\mu}$ is a critical point for $E_\mu|_{\cP_{a,\mu}}$, then $u$ is a critical point for $E_\mu|_{S_a}$. Furthermore:
\begin{itemize}
\item[1)] If $2<q<2+4/N$, then $\cP_0^{a,\mu}= \emptyset$, both $\cP_+^{a,\mu}$ and $\cP_-^{a,\mu}$ are not empty, and 
\[
m(a,\mu) = \min_{\cP_+^{a,\mu}} E_\mu = \min_{\cP_{a,\mu}} E_\mu < 0
\]
\item[2)] If $2+4/N \le q <2^*$, then $\cP_+^{a,\mu} = \cP_0^{a,\mu}= \emptyset$, and 
\[
m(a,\mu) = \min_{\cP_-^{a,\mu}} E_\mu =  \min_{\cP_{a,\mu}} E_\mu \in \left(0, \frac{\cS^\frac{N}2}N\right).
\]
\end{itemize}
\end{proposition}

This explains the discontinuity of the ground state level $m(a,\mu)$ as $q$ passes from the $L^2$-subcritical to the $L^2$-critical or supercritical setting: the introduction of the subcritical perturbation creates a new component $\cP_+^{a,\mu}$ in $\cP_{a,\mu}$, and the minimizer on $\cP_{a,\mu}$ is achieved exactly on this component.

In the same way, Proposition \ref{prop: struct P} also gives a new interpretation of Theorem \ref{thm: mu to 0}: for the homogeneous problem $\mu=0$, it is not difficult to prove that $\cP_{a,0} = \cP^{a,0}_-$, that $\inf_{\cP_{a,0}} E_0 = \cS^\frac{N}2/N$, and that the infimum is achieved by a ground state if and only if $N \ge 5$ (see Proposition \ref{prop: hom} below). Therefore, the introduction of a $L^2$-subcritical perturbation creates a discontinuity for the ground state energy level $m(a,\mu)$ as $\mu \to 0^+$ whenever $N \ge 5$, and more in general it creates a discontinuity of the level $\inf_{\cP_{a,\mu}} E_\mu$ as $\mu \to 0^+$, in any dimension. This phenomenon was already observed in the Sobolev subcritical case (see \cite{So}), and also, in a different form, in \cite{BeJe}, where a different equation is considered, and the discontinuity was created by the introduction of a trapping potential.

%Finally, Proposition \ref{prop: struct P} is a further indication of the fact that, for $\mu<0$, ground states do not exist.

\begin{remark}
As already observed in \cite{So}, the change of the topology in $\cP_{a,\mu}$ obtained by the introduction of a focusing $L^2$-subcritical perturbation reflects what happens to the Nehari manifold in inhomogeneous elliptic problems \cite{Tar}, or in elliptic problems with concave-convex nonlinearities \cite{AmbBreCer, GarPer}. This is somehow surprising, since in \eqref{stat com} all the power-nonlinearities are super-linear; the phenomenon is a direct consequence of the $L^2$-constraint $S_a$, and of the behavior of $E_\mu$ with respect to $L^2$-norm-preserving dilations. Similar ``concave" effects in superlinear problems with $L^2$-constraint were already observed in \cite[Theorem 1.6]{BeJe}, and in the main results of \cite{GoJe, JeLuWa}.
\end{remark}

Lastly, we consider the impact of the study of the function $\Psi_u^\mu$ on the occurrence of finite-time blow-up. Following the same strategy developed in \cite[Theorem 1.13]{So}, we can easily obtain the following result:

\begin{theorem}\label{thm: gwp}
Let us assume that the assumptions of Theorem \ref{thm: main} are satisfied. Let $u \in S_a$ be such that $E_\mu(u) < \inf_{\cP_-^{a,\mu}} E_\mu$. Then $\Psi^\mu_u$ has a unique global maximum point $t_{u,\mu}$, and, if $t_{u,\mu}<0$ and $|x| u \in L^2(\R^N,\C)$, then the solution $\psi$ of \eqref{com nls} with initial datum $u$ blows-up in finite time.
\end{theorem}

As immediate consequence:

\begin{corollary}\label{cor: gwp}
Assume that the assumptions of Theorems \ref{thm: main} are satisfied, and, for $u \in S_a$, let $\psi_u$ be the solution to \eqref{com nls} with initial datum $u$. We have:
\begin{itemize}
\item[1)] If $|x| u \in L^2(\R^N,\C)$, then there exist $\bar s \in \R$ such that
\[
s> \bar s \quad \implies \quad \text{$\psi_{s \star u}$ blows-up in finite time}.
\]
\item[2)] If $2+4/N \le q < 2^*$, $|x| u \in L^2(\R^N,\C)$, $E_\mu(u) <  m(a,\mu)$, and $P_\mu(u)<0$, then $\psi_u$ blows-up in finite time.
\item[3)] If $2<q<2+4/N$, $|x| u \in L^2(\R^N,\C)$ and $E_\mu(u) < m(a,\mu)$, then $\psi_u$ blows-up in finite time.
\end{itemize}
\end{corollary}

With respect to \cite{So}, in Theorem \ref{thm: gwp} we cannot ensure global existence if $t_{u,\mu}>0$. This is another consequence of the fact that the classical blow-up alternative seems unavailable in the Sobolev critical setting. 

For other results regarding global existence, finite time blow-up, scattering and other dynamical properties in the Sobolev critical setting, we refer the reader to \cite{AkaIbrKikNaw3, ChMiZh, KiOhPoVi, MiXuZh, MiaZhaZhe, TaoVisZha} and references therein.

%We start by analyzing the homogeneous limit problem obtained for $\mu=0$.

%\begin{theorem}\label{prop: hom}
%Let $N \ge 3$ and $a>0$, and let $\cP_{a,0}$ be the Pohozaev manifold obtained for the choice $\mu=0$. Then
%\[
%\inf_{u \in \cP_{a,0}} E_0(u) = \frac{1}{N} \cS^\frac{N}2.
%\]
%Therefore:
%\begin{itemize}
%\item[1)] If $N \ge 5$, then $E_0$ on $S_a$ has a unique positive radial ground state, given by 
%\[
%U_{\eps}(x) = [N(N-2)]^\frac{N-2}{4} \left( \frac{\eps}{\eps^2 + |x|^2}\right)^\frac{N-2}2
%\] 
%for the unique choice of $\eps>0$ which gives $\|U_\eps\|_{L^2(\R^N)} = a$. The function $U_\eps$ solves \eqref{stat com} for $\lambda=\mu=0$, and achieves $\inf_{\cP_{a,\mu}} E_0$. 
%\item[2)] If $N=3,4$, then \eqref{stat com} with $\mu=0$ has no positive solution in $S_a$, for any $\lambda \in \R$, and $\inf_{\cP_{a,\mu}} E_0$ is not achieved.
%\end{itemize}
%\end{theorem}
%The difference between the cases $N =3,4$ and $N \ge 5$ reflects the fact that the extremal functions $U_\eps$ for the Sobolev inequalities are in $L^2(\R^N)$ if and only if $N \ge 5$. It is not clear to us whether in dimension $N \ge 5$

\medskip

\noindent \textbf{Structure of the paper, and notation.} In Section \ref{sec: pre} we collect some preliminary results which will often be used in the rest the paper. Section \ref{sec: compactness} contains the discussion of the compactness of Palais-Smale sequences. The proof of our main existence result, Theorem \ref{thm: main}, is given in Sections \ref{sec: sub}, \ref{sec: cr} and \ref{sec: sup}, which contain the $L^2$-subcritical, $L^2$-critical, and $L^2$-supercritical cases, respectively. The non-existence results in the defocusing case, Theorem \ref{thm: non-ex}, is proved in Section \ref{sec: non ex}. The asymptotic behavior of the ground state energy level as $\mu \to 0^+$, Theorem \ref{thm: mu to 0}, is discussed in Section \ref{sec: mu to 0}. Once that Theorem \ref{thm: main} is established, the proofs of Theorems \ref{thm: inst}, \ref{thm: gwp} and of Proposition \ref{prop: struct P} are very similar to the analogue results in the Sobolev subcritical case \cite[Theorems 1.7 and 1.13]{So}, and hence they are only briefly sketched in Section \ref{sec: additional}. 
\medskip

Regarding the notation, in this paper we deal with both complex and real-valued functions, which will be in both cases denoted by $u, v, \dots$. This should not be a source of misunderstanding. The symbol $\bar u$ will always be used for the complex conjugate of $u$. For $p \ge 1$, the (standard) $L^p$-norm of $u \in L^p(\R^N,\C)$ (or of $u \in L^p(\R^N,\R)$) is denoted by $|u|_p$. We simply write $H$ for $H^1(\R^N,\C)$, and $H^1$ for the subspace of real valued functions $H^1(\R^N,\R)$. Similarly, $H^1_{\rad}$ denotes the subspace of functions in $H^1$ which are radially symmetric with respect to $0$, and $S_{a,r} = H^1_{\rad} \cap S_a$. The symbol $\|\cdot\|$ is used only for the norm in $H$ or $H^1$. Denoting by $^*$ the symmetric decreasing rearrangement of a $H^1$ function, we recall that, if $u \in H$, then $|u| \in H^1$, $|u|^* \in H^1_{\rad}$, with 
\[
|\nabla |u|^*|_2 \le |\nabla |u||_2 \le |\nabla u|_2
\]
(it is well known that the symmetric decreasing rearrangement decreases the $L^2$-norm of gradients; regarding the last inequality for complex valued functions, we refer to \cite[Proposition 2.2]{HajStu}). The symbol $\weak$ denotes weak convergence (typically in $H$ or $H^1$). Capital letters $C, C_1, C_2, \dots$ denote positive constants which may depend on $N$, $p$ and $q$ (but never on $a$ or $\mu$), whose precise value can change from line to line. We also mention that, within a section, after having fixed the parameters $a$ and $\mu$, we often choose to omit the dependence of $E_\mu$, $S_a$, $P_\mu$, $\cP_{a,\mu}$, \dots on these quantities, writing simply $E$, $S$, $P$, $\cP$, \dots.

\section{Preliminaries}\label{sec: pre}

In this section we collect several results which will be often used throughout the rest of the paper.

\subsection*{Sobolev inequality and Gagliardo-Nirenberg inequality} For every $N \ge 3$, there exists an optimal constant $\cS>0$ depending only on $N$ such that
\begin{equation}\label{Sob ineq}
\cS |u|_{2^*}^2 \le |\nabla u|_2^2 \qquad \forall u \in \mathcal{D}^{1,2}(\R^N) \qquad \text{(Sobolev inequality)}
\end{equation}
where $\mathcal{D}^{1,2}(\R^N)$ denotes the completion of $C^\infty_c(\R^N)$ with respect to the norm $\|u\|_{\mathcal{D}^{1,2}} := |\nabla u|_2$. It is well known \cite{Tal} that the optimal constant is achieved by (any multiple of) 
\begin{equation}\label{def bubble}
U_{\eps,y}(x) = [N(N-2)]^\frac{N-2}{4} \left( \frac{\eps}{\eps^2 + |x-y|^2}\right)^\frac{N-2}2, \qquad \eps>0, \ y \in \R^N,
\end{equation}
which are the only positive classical solutions to the critical Lane-Emden equation
\begin{equation}\label{cr LE}
-\Delta w = w^{2^*-1}, \quad w>0 \qquad \text{in $\R^N$}.
\end{equation}

If $p \in (2,2^*)$, we also recall that there exists an optimal constant $C_{N,p}$ depending on $N$ and on $p$ such that
\begin{equation}\label{GN ineq}
|u|_p \le C_{N,p} |\nabla u|_2^{\gamma_p} |u|_2^{1-\gamma_p} \qquad \forall u \in H \qquad \text{(Gagliardo-Nirenberg inequality)}
\end{equation}
where $\gamma_p$ is defined by \eqref{def gamma_p}.
%\begin{equation}\label{def gamma_p}
%\gamma_p = \frac{N(p-2)}{2p}.
%\end{equation}
It is convenient to observe that
\[
\gamma_p p \begin{cases} <2 & \text{if }2<p<\bar p \\=2 & \text{if $p= \bar p$} \\
>2 & \text{if $\bar p < p < 2^*$},
\end{cases} \quad \text{and that} \quad \gamma_{2^*} = 1.
\]
We can then see \eqref{Sob ineq} as a particular case of \eqref{GN ineq}, with $C_{N, 2^*} = \cS^{-1/2}$.

\subsection*{Behavior of $E_\mu$ with respect to dilations.} As in most of the papers regarding normalized solutions of Schr\"odinger-type equations in $\R^N$, the study of the behavior of $E_\mu$ with respect to the $L^2$-norm preserving variations is crucial. We consider, for $u \in S_a$ and $s \in \R$, the fiber $\Psi_u^\mu$ introduced in \eqref{def psi}, and note that
\begin{equation*}
\begin{split}
(\Psi^\mu_u)'(s) &=  e^{2s}\int_{\R^N} |\nabla u|^2 - e^{2^*  s} \int_{\R^N} |u|^{2^*} -  \mu \gamma_q e^{q \gamma_q  s} \int_{\R^N} |u|^q \\
& = \int_{\R^N} |\nabla (s \star u)|^2 -  \int_{\R^N} |s \star u|^{2^*} - \mu \gamma_q \int_{\R^N} |s \star u|^q = P_\mu(s \star u),
\end{split}
\end{equation*}
where $P_\mu$ is defined by \eqref{Poh}. Therefore, as anticipated in the introduction, we have:
\begin{proposition}\label{prop: psi P}
Let $u \in S_a$. Then $s \in \R$ is a critical point for $\Psi^\mu_u$ if and only if $s \star u \in \cP_{a,\mu}$. 
\end{proposition}
In particular, $u \in \cP_{a,\mu}$ if and only if $0$ is a critical point of $\Psi_u^\mu$: $\cP_{a,\mu} = \{u \in S_a: \Psi_u'(0)=0\}$. For future convenience, we also recall (see \cite[Lemma 3.5]{BaSo2}) that the map
\begin{equation}\label{cont star}
(s,u) \in \R \times H^1 \mapsto (s \star u) \in H^1 \quad \text{is continuous}.
\end{equation}

\subsection*{The homogeneous Sobolev critical NLSE} We focus here on the case $\mu=0$, and in particular to existence and properties of ground states for 
\[
E_0(u) = \int_{\R^N} \left(\frac12 |\nabla u|^2 - \frac1{2^*} |u|^{2^*}\right)
\]
on $S_a$. The associated Pohozaev manifold is 
\[
\cP_{a,0} = \{u \in S_a: \ |\nabla u|_2^2 = |u|_{2^*}^{2^*}\} = \{u \in S_a: (\Psi_u^0)'(0) = 0\}
\]
with $\Psi_u^0$ defined by \eqref{def psi} for $\mu=0$:
\begin{equation}\label{def psi 0}
\Psi_u^0(s) = E_0(s \star u) = \frac{e^{2s}}{2} |\nabla u|_2^2 - \frac{e^{2^* s}}{2^*} |u|_{2^*}^{2^*}.
\end{equation}
We also recall the decomposition $\cP_{a,0} = \cP_+^{a,0} \cup \cP_0^{a,0} \cup \cP_{-}^{a,0}$, introduced in \eqref{split P}. 
%The following result, which can be obtained combining some well known facts, already enlighten the peculiarity of the critical case with respect to the subcritical one. 

\begin{proposition}\label{prop: hom}
Let $N \ge 3$ and $a>0$. Then $\cP_{a,0}$ is a smooth manifold of codimension $1$ in $S_a$, $\cP_{a,0} = \cP_{a,0}^-$, and
\begin{equation}\label{min P 0}
\inf_{u \in \cP_{a,0}} E_0(u) = \inf_{u \in S_{a}} \max_{s \in \R} E_0(s \star u) = \frac{1}{N} \cS^\frac{N}2.
\end{equation}
Moreover:
\begin{itemize}
\item[1)] If $N \ge 5$, then the infimum is achieved, and $E_0$ on $S_a$ has a unique positive radial ground state, given by $U_{\eps,0}$ defined in \eqref{def bubble} for the unique choice of $\eps>0$ which gives $\|U_{\eps,0}\|_{L^2(\R^N)} = a$. The function $U_{\eps,0}$ solves \eqref{stat com} for $\lambda=\mu=0$. %and achieves $\inf_{\cP_{a,\mu}} E_0$. 
\item[2)] If $N=3,4$, then \eqref{stat com} with $\mu=0$ has no positive solution in $S_a$, for any $\lambda \in \R$, and in particular $\inf_{\cP_{a,\mu}} E_0$ is not achieved.
\end{itemize}
\end{proposition}
This simple proposition enlightens the peculiarity of the critical case with respect to the subcritical ones, where the dimension does not play any role (see e.g. Section 2 in \cite{So} for a brief account on them). The difference between the cases $N =3,4$ and $N \ge 5$ reflects the fact that functions $U_{\eps,y}$ (the extremal functions for the Sobolev inequality) are in $L^2(\R^N)$ if and only if $N \ge 5$. Notice that, if $N=3,4$, the above proposition implies the non-existence of positive ground states.

\begin{proof}
The proof of the proposition can be easily obtained combining some well known facts. First of all, it is not difficult to check that for every $u \in S_a$ the function $\Psi_u^0$ has a unique critical point $t_{u,0}$, which is a strict maximum point, and is given by 
\begin{equation}\label{def t_0 u}
e^{t_{u,0}}= \left( \frac{|\nabla u|_2^2}{|u|_{2^*}^{2^*}} \right)^\frac{1}{2^*-2}. 
\end{equation}
This implies, by Proposition \ref{prop: psi P}, that $\cP_+^{a,0}= \emptyset$. Also, if $u \in \cP_0^{a,0}$, then combining $(\Psi_u^0)'(0) = 0= (\Psi_u^0)''(0)$ we obtain
\[
2|\nabla u|_2^2 = 2^*|u|_{2^*} = 2^* |\nabla u|_2^2 \quad \implies \quad |\nabla u|_2 =0,
\]
which is not possible since $u \in S_a$. Then $\cP_{a,0}=\cP^{a,0}_-$. Moreover, using that $\cP_0^{a,0}= \emptyset$, it is standard to show that $\cP_{a,0}$ is a smooth manifold of codimension $1$ on $S_a$ (see e.g. \cite[Lemma 5.2]{So} for an analogue proof), and that $\cP_{a,0}$ is a natural constraint: if $u \in \cP_{a,0}$ is a critical point of $E_0|_{\cP_{a,0}}$, then $u$ is a critical point of $E_0|_{S_a}$ (see e.g. \cite[Proposition 1.11]{So}). 

The proof of the first equality in \eqref{min P 0} is contained in the forthcoming Lemma \ref{inf sup}, which regards the more general case $\mu \ge 0$. Now, by \eqref{def t_0 u}
\begin{align*}
\inf_{u \in \cP_{a,0}} E_0(u) &= \inf_{u \in S_a} \max_{s \in \R} E_0(s \star u) = \inf_{u \in S_a} \left[\frac{1}{2}\left( \frac{|\nabla u|_2^2}{|u|_{2^*}^{2^*}} \right)^\frac{2}{2^*-2} |\nabla u|_2^2 - \frac{1}{2^*}  \left( \frac{|\nabla u|_2^2}{|u|_{2^*}^{2^*}} \right)^\frac{2^*}{2^*-2} |u|_{2^*}^{2^*} \right] \\
& = \inf_{u \in S_a} \frac{1}{N} \left( \frac{|\nabla u|_2^2}{|u|_{2^*}^{2}} \right)^\frac{2^*}{2^*-2} = \inf_{u \in H^1(\R^N) \setminus \{0\}} \frac{1}{N} \left( \frac{|\nabla u|_2^2}{|u|_{2^*}^{2}} \right)^\frac{N}{2},
\end{align*}
that is, the minimization of $E_0$ on $\cP_{a,0}$ is equivalent to the minimization of the Sobolev quotient in $H^1(\R^N) \setminus \{0\}$. By density of $H^1$ in $\cD^{1,2}$, we infer that the infimum is $\cS^{\frac{N}2}/N$, and is achieved if and only if the extremal functions $U_{\eps,y}$ defined in \eqref{def bubble} stay in $L^2(\R^N)$, namely if and only if $N \ge 5$. Recalling that any critical point of $E_0$ on $S_a$ stays in $\cP_{a,0}$, and using the fact that $\cP_{a,0}$ is a natural constraint, the case $N \ge 5$ follows. If instead $N=3,4$, we can show that the infimum of $E_0$ on $\cP_{a,0}$ is not achieved. Let us assume by contradiction that there exists a minimizer $u$, and let $v := |u|^*$, the symmetric decreasing rearrangement of $u$. By standard properties $|\nabla v|_2 \le |\nabla u|_2$, $E_0(v) \le E_0(u)$, and $P_0(v) \le 0=P_0(u)$. If $P_0(v)<0$, then $t_{v,0}$ defined by \eqref{def t_0 u} is negative, and hence
\[
E_0(u) \le E_0(t_{v,0} \star v) = \frac{e^{2 t_{v,0}}}{N} |\nabla v|_2^2 \le \frac{e^{2 t_{v,0}}}{N} |\nabla u|_2^2 =   e^{2 t_{v,0}} E_0(u) < E_0(u),
\]
a contradiction. Thus, it is necessary that $P_0(v) = 0$, that is $v \in \cP_{a,0}$, and $v$ is a nonnegative radial minimizer. Since $\cP_{a,0}$ is a natural constraint, we deduce that
\begin{equation}\label{20jun19}
-\Delta v= \lambda v +  v^{2^*-1}, \quad v \ge 0 \quad \text{in $\R^N$}
\end{equation}
for some $\lambda \in \R$. Since $P_0(v) = 0$, necessarily $\lambda = 0$. Also, $v>0$ in $\R^N$ by the strong maximum principle. Hence, by \cite[Corollary 8.2]{CafGidSpr}, we have that $v = \alpha U_{\eps,0}$ for some $\alpha, \eps >0$. This is however not possible, since $U_{\eps,0} \not \in H^1(\R^N)$ for $N=3,4$. Notice that we may still hope to find a positive ground state, that is, a positive critical point of $E_0$ on $S_a$ minimizing the energy among all the critical points. But this would be a solution to \eqref{20jun19}, for some $\lambda \in \R$, and arguing as before it is not difficult to check that this is not possible. 
\end{proof}

\section{Compactness of Palais-Smale sequences}\label{sec: compactness}

%In the search of normalized solutions for $L^2$-supercritical problems with a mass constraint, the compactness of Palais-Smale sequence is a major issue. This is already true in a Sobolev subcritical framework, as observed in \cite{So}: indeed, the boundedness of a PS sequence is not guaranteed in general; sequences of approximated Lagrange multipliers have to be controlled; and moreover, weak limits of PS sequence could leave the constraint, since the embeddings $H^1(\R^N) \hookrightarrow L^2(\R^N)$ and also $H^1_{\rad}(\R^N) \hookrightarrow L^2(\R^N)$ are not compact. As expected, dealing with a Sobolev critical problem makes the compactness task even more challenging.

In this section we discuss the convergence of Palais-Smale sequences (PS sequences for short) satisfying suitable additional conditions. This is probably one of the most delicate ingredient in the proofs of our main results, as already mentioned in the introduction.

\begin{proposition}\label{prop: PS conv}
Let $N \ge 3$, $2 <q < 2^*$, and let $a,\mu>0$. Let $\{u_n\} \subset S_{a,r} = S_a \cap H^1_{\rad}$ be a Palais-Smale sequence for $E_\mu|_{S_{a}}$ at level $m$, with 
\[
m < \frac{\cS^{\frac{N}{2}}}N \quad \text{and} \quad m \neq 0,
\]
where $\cS$ denotes the best constant in the Sobolev inequality. Suppose in addition that $P_\mu(u_n) \to 0$ as $n \to \infty$. Then one of the following alternatives holds:
\begin{itemize}
\item[$i$)] either up to a subsequence $u_n \weak u$ weakly in $H^1(\R^N)$ but not strongly, where $u \not \equiv 0$ is a solution to \eqref{stat com} for some $\lambda<0$, and 
\[
E_\mu(u) \le m - \frac{\cS^{\frac{N}{2}}}N;
\]
\item[$ii$)] or up to a subsequence $u_n \to u$ strongly in $H^1(\R^N)$, $E_\mu(u) = m$, and $u$ solves \eqref{stat com}-\eqref{norm} for some $\lambda<0$.\end{itemize}
\end{proposition}

The proof is a sort of combination of the techniques introduced by H. Brezis and L. Nirenberg in \cite{BreNir}, and by L. Jeanjean in \cite{Jea}. Throughout this section, we assume that the assumptions of Proposition \ref{prop: PS conv} hold, and, since $a$ and $\mu$ are fixed, we omit the dependence on $E_\mu$, $P_\mu$, $S_a$, \dots on these quantities, writing simply $E$, $P$, $S$, \dots.

\begin{lemma}\label{lem: bdd}
The sequence $\{u_n\}$ is bounded in $H^1(\R^N)$.
\end{lemma}

\begin{proof}
We assume at first that $q<2+4/N$, so that $\gamma_q q<2$. Since $P(u_n) \to 0$, we have
\[
\begin{split}
E(u_n) &= \frac1N|\nabla u_n|_2^2 - \frac{\mu}q \left( 1-\frac{\gamma_q q}{2^*}\right) |u_n|_q^q + o(1) \\
& \ge \frac1N|\nabla u_n|_2^2 - \frac{\mu}q C_{N,q}^q \left( 1-\frac{\gamma_q q}{2^*}\right) a^{(1-\gamma_q)q} |\nabla u_n|_2^{\gamma_q q} + o(1),
\end{split}
\]
by the Gagliardo-Nirenberg inequality. Then, using also that $E(u_n) \le m+1$ for $n$ large, we deduce that
\[
\frac1N|\nabla u_n|_2^2 \le \frac{\mu}q C_{N,q}^q \left( 1-\frac{\gamma_q q}{2^*}\right) a^{(1-\gamma_q)q} |\nabla u_n|_2^{\gamma_q q} + m +2,
\]
and this implies that $\{u_n\}$ is bounded.

Let now $q = \bar p =2+ 4/N$, and recall that $\gamma_{\bar p} \bar p=2$. Then $P(u_n) \to 0$ gives
\[
m+1 \ge E(u_n) = \frac{1}{N} |u_n|_{2^*}^{2^*} +o(1) \quad \implies \quad |u_n|_{2^*}^{2^*} \le C 
\]
for every $n$ large. By the H\"older inequality $|u_n|_q^q \le |u_n|_2^{2(1-\alpha)} |u_n|_{2^*}^{2^*\alpha} \le C$ as well (for a suitable $\alpha \in (0,1)$), and, using again that $P(u_n) \to 0$, we obtain
\[
|\nabla u_n|_2^2 = \mu \gamma_q |u_n|_q^q + |u_n|_{2^*}^{2^*} + o(1)\le C,
\]
as desired.

Finally, let $2+4/N < q<2^*$, so that $\gamma_q q >2$. Since $P(u_n) \to 0$, we have
\[
E(u_n) = \frac{\mu}q\left( \frac{\gamma_q q}2-1\right) |u_n|_q^q + \frac{1}{N} |u_n|_{2^*}^{2^*} + o(1),
\]
and the coefficient of $|u_n|_q^q$ is positive. Therefore, the boundedness of $E(u_n)$ implies that $\{|u_n|_q\}$ and $\{|u_n|_{2^*}\}$ are both bounded, and in turn this implies that $\{|\nabla u_n|_2\}$ is bounded, since $P(u_n) \to 0$.
\end{proof}

\begin{proof}[Proof of Proposition \ref{prop: PS conv}]
By the previous lemma, the sequence $\{u_n\}$ is a bounded sequence of radial functions, and hence, by compactness of $H^1_{\rad}(\R^N) \hookrightarrow L^q(\R^N)$, there exists $u \in H^1_{\rad}(\R^N)$ such that up to a subsequence $u_n \weak u$ weakly in $H^1(\R^N)$, $u_n \to u$ strongly in $L^q(\R^N)$, and a.e. in $\R^N$. Letting $v_n:= u_n -u$, this means that $v_n \weak 0$ weakly in $H^1(\R^N)$, $v_n \to 0$ strongly in $L^q(\R^N)$, and a.e. in $\R^N$. Since $\{u_n\}$ is a bounded PS sequence for $E|_S$,  there exists $\{\lambda_n\} \subset \R$ such that for every $\varphi \in H^1(\R^N)$
\begin{equation}\label{eq u_n}
\int_{\R^N} \nabla u_n \cdot \nabla \varphi - \lambda_n u_n \varphi - \mu |u_n|^{q-2} u_n \varphi - |u_n|^{2^*-2} u_n \varphi = o(1) \|\varphi\|
\end{equation}
as $n \to \infty$, by the Lagrange multipliers rule. Choosing $\varphi = u_n$, we deduce that $\{\lambda_n\}$ is bounded as well, and hence up to a subsequence $\lambda_n \to \lambda \in \R$. Using the fact that $P(u_n) \to 0$ and $\gamma_q<1$ we deduce that
\begin{equation}\label{lamb le 0}
\begin{split}
\lambda a^2 &= \lim_{n \to \infty} \lambda_n |u_n|_2^2 = \lim_{n \to \infty} \left( |\nabla u_n|_2^2 -\mu |u_n|_q^q - |u_n|_{2^*}^{2^*}\right)  \\
& =\lim_{n \to \infty}  \mu(\gamma_q -1) |u_n|_q^q = \mu(\gamma_q -1) |u|_q^q  \le 0,
\end{split}
\end{equation}
with $\lambda = 0$ if and only if $u \equiv 0$. We claim now that
\begin{equation}\label{lem: u not 0}
\text{the weak limit $u$ does not vanish identically}.
\end{equation}
Suppose by contradiction that $u \equiv 0$. Since $\{u_n\}$ is bounded in $H^1(\R^N)$, up to a subsequence we have that $|\nabla u_n|_2^2 \to \ell \in \R$. But $P(u_n) \to 0$ and $u_n \to 0$ strongly in $L^q$, and hence
\[
|u_n|_{2^*}^{2^*} = |\nabla u_n|_2^2 - \mu \gamma_q |u_n|_q^q \to \ell
\]
as well. Therefore, by the Sobolev inequality $\ell \ge S \ell^\frac{2}{2^*}$, and we deduce that 
\[
\text{either $\ell = 0$, or $\ell \ge \cS^{\frac{N}2}$}.
\]
Let us suppose at first that $\ell \ge \cS^{N/2}$. Since $E(u_n) \to m$ and $P(u_n) \to 0$, we have that
\[
\begin{split}
m +o(1) &= E(u_n) = \frac{1}{N} |\nabla u_n|_2^2 - \frac{\mu}{q} \left( 1 - \frac{\gamma_q q}{2^*}\right) |u_n|_q^q +o(1) \\
&= \frac{1}{N} |\nabla u_n|_2^2 + o(1) = \frac{\ell}{N} + o(1),
\end{split}
\]
whence $m = \ell/N$. However, this is not possible since in this case $m \ge \cS^{N/2}/N$, which contradicts our assumptions. If instead $\ell=0$, we have $|u_n|_q \to 0$, $|\nabla u_n|_2 \to 0$, and $|u_n|_{2^*} \to 0$. But then $E(u_n) \to 0$, which gives again a contradiction. Thus, claim \eqref{lem: u not 0} is verified. 

As already observed, \eqref{lamb le 0} and \eqref{lem: u not 0} imply that $\lambda<0$. Moreover, passing to the limit in \eqref{eq u_n} by weak convergence, we obtain
\begin{equation}\label{eq u lim}
-\Delta u= \lambda u + \mu|u|^{q-2} u + |u|^{2^*-2} u \qquad \text{in $\R^N$},
\end{equation}
and hence by the Pohozaev identity $P(u) = 0$. Recalling that $v_n = u_n-u \weak 0$ in $H^1(\R^N)$, we also have
\begin{equation}\label{6111}
|\nabla u_n|_2^2 = |\nabla u|_2^2 + |\nabla v_n|_2^2 + o(1),
\end{equation}
and by the Brezis-Lieb lemma \cite{BrLi}
\begin{equation}\label{6112}
|u_n|_{2^*}^{2^*} = |u|_{2^*}^{2^*} + |v_n|_{2^*}^{2^*} + o(1).
\end{equation}
Therefore, since $P(u_n) \to 0$, and $u_n \to u$ strongly in $L^q$, we deduce that
\[
|\nabla u|_2^2 + |\nabla v_n|_2^2 = \mu \gamma_q |u|_q^q +  |u|_{2^*}^{2^*} + |v_n|_{2^*}^{2^*} + o(1).
\]
But $P(u) = 0$, and hence $|\nabla v_n|_2^2 = |v_n|_{2^*}^{2^*} + o(1)$. We infer that up to a subsequence
\[
\lim_{n \to \infty} |\nabla v_n|_2^2 = \lim_{n \to \infty} |v_n|_{2^*}^{2^*} = \ell \ge 0, \quad \implies \quad \ell \ge S \ell^\frac{2}{2^*}
\]
by the Sobolev inequality. Therefore, either $\ell = 0$, or $\ell \ge \cS^{\frac{N}2}$. If $\ell \ge \cS^{N/2}$, then by \eqref{6111} and \eqref{6112}
\[
m = \lim_{n \to \infty} E(u_n) = 	\lim_{n \to \infty} \left( E(u) + \frac12 |\nabla v_n|_2^2 - \frac1{2^*}|v_n|_{2^*}^{2^*}  \right) = E(u) + \frac{\ell}{N} \ge E(u) + \frac{\cS^{\frac{N}{2}}}N,
\]
whence alternative ($i$) in the thesis of the proposition follows. If instead $\ell= 0$, then we show that $u_n \to u$ strongly in $H^1$. Indeed, $\|u_n-u\|_{\cD^{1,2}} = |\nabla v_n|_2 \to 0$ establishes that $u_n \to u$ strongly in $\mathcal{D}^{1,2}(\R^N)$, and hence in $L^{2^*}(\R^N)$ by the Sobolev inequality. In order to prove that $u_n \to u$ strongly in $L^2$, we test \eqref{eq u_n} with $\varphi = u_n-u$, test \eqref{eq u lim} with $u_n-u$, and subtract, obtaining
\begin{multline*}
\int_{\R^N} |\nabla (u_n-u)|^2 -\int_{\R^N} \left( \lambda_n u_n-  \lambda u \right) (u_n-u) = \\ \int_{\R^N} \left( |u_n|^{q-2} u_n - |u|^{q-2} u\right)(u_n-u) +\int_{\R^N} \left( |u_n|^{2^*-2} u_n - |u|^{2^*-2} u\right)  (u_n-u) + o(1).
\end{multline*}
Now the first, the third, and the fourth integrals tend to $0$ by convergence in $\mathcal{D}^{1,2}$, $L^q$, and $L^{2^*}$. As a consequence
\[
0 = \lim_{n \to \infty} \int_{\R^N} \left( \lambda_n u_n-  \lambda u \right) (u_n-u)  =  \lim_{n \to \infty} \lambda \int_{\R^N} (u_n -u)^2,
\]
and this completes the proof.
% shows that if $\ell=0$, then alternative ($ii$) in the thesis of the proposition holds.
\end{proof}

We conclude this section stating the following variant of Proposition \ref{prop: PS conv}. 

\begin{proposition}\label{prop: PS conv 2}
Let $N \ge 3$, $2 <q < 2^*$, and let $a,\mu>0$. Let $\{u_n\} \subset S_{a}$ be a Palais-Smale sequence for $E_\mu|_{S_{a}}$ at level $m$, with 
\[
m < \frac{\cS^{\frac{N}{2}}}N \quad \text{and} \quad m \neq 0.
\]
Suppose in addition that $P_\mu(u_n) \to 0$ as $n \to \infty$, and that there exists $\{v_n\} \subset S_{a}$, with $v_n$ radially symmetric for every $n$, such that $\|u_n-v_n\| \to 0$ as $n \to \infty$. Then one of the alternatives ($i$) and ($ii$) in Proposition \ref{prop: PS conv} holds.
%
%
%
%following alternative hold:
%\begin{itemize}
%\item[$i$)] either up to a subsequence $u_n \weak u$ weakly in $H^1(\R^N)$, where $u \not \equiv 0$ is a solution to \eqref{stat com} for some $\lambda<0$, and 
%\[
%E_\mu(u) \le m - \frac{\cS^{\frac{N}{2}}}N
%\]
%(but the $L^2$-norm of $u$ is not necessarily equal to $a$);
%\item[$ii$)] or up to a subsequence $u_n \to u$ strongly in $H^1(\R^N)$, $E_\mu(u) = m$, and $u$ solves \eqref{stat com}-\eqref{norm} for some $\lambda<0$.\end{itemize}
\end{proposition}

The proof is analogue to the previous one: as in Lemma \ref{lem: bdd}, we show that $\{u_n\}$ is bounded. Then also $\{v_n\}$ is bounded, and, since each $v_n$ is radial, we deduce that $v_n \weak u$ weakly in $H^1$, $v_n \to u$ strongly in $L^q$, and a.e. in $\R^N$, up to a subsequence. Since $\|u_n-v_n\| \to 0$, the same convergence is inherited by $\{u_n\}$, and we can proceed as in the proof of Proposition \ref{prop: PS conv}.

%\begin{remark}\label{rmk: sup sup}
%The proof of Lemmas \ref{lem: conv PS subcr} and \ref{lem: conv PS subcr N=1} can be easily extended to the case $\bar p<q<p$, $\mu<0$, which was not considered so far in the literature. The only point where we need to modify the proof is estimate \eqref{23101}, since $\gamma_q q>2$. However, using the boundeness of $\{u_n\}$, it is easy to see that $\lambda<0$ if $\mu$ and $a$ satisfy a condition similar to \eqref{hp mu < 0 conv}. We omit the details.
%\end{remark}

\section{$L^2$-subcritical perturbation}\label{sec: sub}

For $N \ge 3$ and $2<q < 2+4/N$, let
\[
C' :=  \left(\frac{2^* \cS^{\frac{2^*}2} (2-\gamma_q q)}{2(2^*-\gamma_q q)}  \right)^{\frac{2-\gamma_q q}{2^*-2}}\left(\frac{q(2^*-2)}{2C_{N,q}^q(2^*-\gamma_q q)}\right),
\]
\[
\begin{split}
C'' :
%&= \left(\frac{2q \, 2^* \cS^\frac{N}2}{N  C_{N,q}^q (2^*-\gamma_q q)(2-\gamma_q q)}\right)^\frac{2-\gamma_q q}{2} \left(\frac{2 2^*}{N \gamma_q C_{N,q}^q(2^*-\gamma_q q)}\right)^\frac{\gamma_q q}{2}\\
&= \frac{2 \, 2^*}{N \gamma_q C_{N,q}^q(2^*-\gamma_q q)} \left( \frac{\gamma_q q \, \cS^{\frac{N}2}}{2-\gamma_q q}\right)^\frac{2-\gamma_q q}2.
%\frac{1}{N \gamma_q} \left( \frac{\gamma_q \cS^{\frac{N}2} }{2-\gamma_q q}\right)^{\frac{2-\gamma_q q}{2}} \left( \frac{2 \, 2^*}{C_{N,q}^q (2^*-\gamma_q q)}\right).
\end{split}
\]
In this section we prove that, for $2<q<2+4/N$, Theorem \ref{thm: main} holds for any $a,\mu>0$ satisfying \eqref{hp} with
\begin{equation}\label{alpha sub}
\alpha(N,q):= \min\{C', C''\}.
\end{equation}
Throughout the proof, $a$ and $\mu$ satisfying \eqref{hp} with this definition of $\alpha$ will be fixed, and hence we often omit the dependence on these quantities.

The geometry of the constrained functional $E|_{S}$ is very similar to the one associated with \eqref{com nls} in case $2<q<2+4/N<p<2^*$, which was carefully analyzed in \cite[Section 5]{So}. As a first observation, we note that by the Gagliardo-Nirenberg and the Sobolev inequalities
\begin{equation}\label{E sub sup}
E (u) \ge \frac{1}{2} |\nabla u|_2^2 - \mu \frac{C_{N,q}^q}{q} a^{(1-\gamma_q) q} |\nabla u|_2^{\gamma_q q} -  \frac{1}{2^* \cS^{\frac{2^*}2}} |\nabla u|_2^{2^*},
\end{equation}
for every $u \in S$. Therefore, we consider the function $h : \R^+ \to \R$
\begin{equation*}%\label{def h}
h(t):=  \frac{1}{2} t^2 - \mu \frac{C_{N,q}^q}{q} a^{(1-\gamma_q) q} t^{\gamma_q q} -  \frac{1}{2^* \cS^{\frac{2^*}2}} t^{2^*}.
\end{equation*}
Since $\mu>0$ and $\gamma_q q<2<2^*$, we have that $h(0^+) = 0^-$ and $h(+\infty) = -\infty$. Moreover:

\begin{lemma}\label{lem: struct h}
Under assumption \eqref{hp}, the function $h$ has a local strict minimum at negative level, a global strict maximum at positive level, and no other critical points, and there exist 
%\[
%0<R_0< \left( \frac{2^* \cS^{\frac{2^*}2} (2-\gamma_q q)}{2(2^*-\gamma_q q)} \right)^{\frac{1}{2^*-2}}<R_1,
%\]
%with 
$R_0$ and $R_1$, both depending on $a$ and $\mu$, such that $h(R_0) = 0 = h(R_1)$ and $h(t) >0$ iff $t \in (R_0,R_1)$. 
\end{lemma}
\begin{proof}
By assumption $\mu a^{(1-\gamma_q)q} < C'$, and, recalling that $\gamma_{2^*}=1$ and $C_{N,2^*}= \cS^{-1/2}$, it is immediate to see that this condition coincides with \cite[Assumption (1.6)]{So} in case $p=2^*$. Therefore, we can proceed exactly as in \cite[Lemma 5.1]{So}. 
\end{proof}

Using again that $\mu a^{(1-\gamma_q)q}< C'$, we can also prove as in \cite[Lemma 5.2]{So} that $\cP_0 = \emptyset$, and $\cP$ is a smooth manifold of codimension $1$ in $S$. This fact can in turn be used in the following lemma.

\begin{lemma}\label{lem: fiber sub}
For every $u \in S$, the function $\Psi_u$ has exactly two critical points $s_u< t_u \in \R$ and two zeros $c_u <d_u \in \R$, with $s_u<c_u<t_u<d_u$. Moreover:
\begin{itemize}
\item[1)] $s_u \star u \in \cP_+$, and $t_u \star u \in \cP_-$, and if $s \star u \in \cP$, then either $s =s_u$ or $s=t_u$. 
\item[2)] $|\nabla (s \star u)|_2 \le R_0$ for every $s \le c_u$, and
\[
E(s_u \star u) = \min \left\{E(s \star u): \ \text{$s \in \R$ and $|\nabla (s \star u)|_2 < R_0$}\right\} < 0.
\]
\item[3)] We have
\[
E(t_u \star u) = \max \left\{E(s \star u): \ s \in \R\right\} > 0,
\]
and $\Psi_u$ is strictly decreasing and concave on $(t_u,+\infty)$. In particular, if $t_u<0$, then $P(u)<0$.
\item[4)] The maps $u \in S \mapsto s_u \in \R$ and $u \in S \mapsto t_u \in \R$ are of class $C^1$.
%\item[5)] If $t_u<0$, then $P(u) <0$.
\end{itemize}
\end{lemma}
Again, the proof is completely analogue to the one in \cite[Lemma 5.3]{So}. 

Now, for $k >0$, we set
\[
A_k:= \left\{u \in S: |\nabla u|_2<k\right\}, \quad \text{and} \quad m(a,\mu):= \inf_{u \in A_{R_0}} E(u).
\]
From Lemma \ref{lem: fiber sub} it follows that $\cP_+\subset A_{R_0}$, and $\sup_{\cP_+}E \le 0 \le \inf_{\cP_-}E$. Moreover, as in \cite[Lemma 5.5]{So} we have that $m(a,\mu) \in (-\infty,0)$, that
\begin{equation}\label{inf interno}
m(a,\mu) =  \inf_{\cP} E = \inf_{\cP_+}E, \quad \text{and that} \quad m(a,\mu) < \, \inf_{\overline{A_{R_0}} \setminus A_{R_0-\rho}} E
\end{equation}
for $\rho>0$ small enough. This allows to proceed with the:

\begin{proof}[Proof of Theorem \ref{thm: main} for $2<q<2+4/N$]
Let $\{v_n\}$ be a minimizing sequence for $\inf_{A_{R_0}} E$. It is not restrictive to assume that $v_n \in S_r$ is radially decreasing for every $n$ (if this is not the case, we can replace $v_n$ with $|v_n|^*$, the Schwarz rearrangement of $|v_n|$, and we obtain another function in $A_{R_0}$ with $E(|v_n|^*) \le E(v_n)$). Furthermore, for every $n$ we can take $s_{v_n} \star v_n \in \cP_+$, observing that then $|\nabla (s_{v_n} \star v_n)|_2 <R_0$ by Lemma \ref{lem: fiber sub}, and
\[
E(s_{v_n} \star v_n) = \min\left\{E(s \star v_n): \ \text{$s \in \R$ and $|\nabla (s \star v_n)|_2 < R_0$}\right\} \le E(v_n);
\]
in this way we obtain a new minimizing sequence $\{w_n=s_{v_n} \star v_n\}$, with $w_n \in S_r \cap \cP_+$ radially decreasing for every $n$. By \eqref{inf interno}, $|\nabla w_n|_2 < R_0-\rho$ for every $n$, and hence Ekeland's variational principle yields in a standard way the existence of a new minimizing sequence $\{u_n\} \subset A_{R_0}$ for $m(a,\mu)$, with the property that $\|u_n-w_n\| \to 0$ as $n \to \infty$, which is also a Palais-Smale sequence for $E$ on $S$. The condition $\|u_n-w_n\| \to 0$, together with the boundedness of $\{w_n\}$ (each $w_n$ stays in $A_{R_0}$), implies that $P(u_n) =P(w_n) +o(1) \to 0$ as $n \to \infty$: indeed
\begin{align*}
\int_{\R^N}|\nabla u_n|^2 &= \int_{\R^N}|\nabla w_n|^2 + \int_{\R^N}|\nabla (u_n- w_n)|^2 + 2 \int_{\R^N} \nabla w_n \cdot \nabla (u_n-w_n) = \int_{\R^N}|\nabla w_n|^2 + o(1) \\
\int_{\R^N}|u_n|^p &= \int_{\R^N} |w_n|^p+ \int_{\R^N} p|w_n+\xi_n (u_n-w_n)|^{p-1} (u_n-w_n) =\int_{\R^N} |w_n|^p + o(1),
\end{align*}
for every $p \in [2,2^*]$, where $\xi_n= \xi_n(x) \in [0,1]$. Hence one of the alternatives in Proposition \ref{prop: PS conv 2} holds. We wish to show that necessarily the second alternative occurs. Assume then by contradiction that $u_n \weak u$ weakly in $H^1$, where $u$ solves \eqref{stat com} for some $\lambda<0$ and 
\[
E(u) \le m(a,\mu) - \frac{\cS^{\frac{N}2}}{N}.
\]
Since $u$ solves \eqref{stat com}, by the Pohozaev identity $P(u) = 0$. Therefore, by the Gagliardo-Nirenberg inequality
\begin{equation}\label{7121}
\begin{split}
m(a,\mu) &\ge \frac{\cS^{\frac{N}2}}{N} + E(u) \ge \frac{\cS^{\frac{N}2}}{N}+ \frac1N|\nabla u|_2^2 -\frac{\mu}q\left( 1- \frac{\gamma_q q}{2^*} \right) |u|_q^q \\
%& \ge  \frac{\cS^{\frac{N}2}}{N}+ \frac1N|\nabla u|_2^2 -\frac{\mu}q\left( 1- \frac{\gamma_q q}{2^*} \right) C_{N,q}^q |u|_2^{(1-\gamma_q)q} |\nabla u|_2^{\gamma_q q} \\
& \ge \frac{\cS^{\frac{N}2}}{N}+ \frac1N|\nabla u|_2^2 -\frac{\mu}q\left( 1- \frac{\gamma_q q}{2^*} \right) C_{N,q}^q a^{(1-\gamma_q)q} |\nabla u|_2^{\gamma_q q},
\end{split}
\end{equation}
where we used the fact that $|u|_2 \le a$ by weak convergence. Now the idea is that if $\mu a^{(1-\gamma_q)q}$ is smaller than $C''$, then the right hand side is positive, in contradiction with the fact that $m(a,\mu) <0$. This can be checked rigorously in the following way: we introduce
\[
\vartheta(t) := \frac1N t^2 -\frac{\mu}q\left( 1- \frac{\gamma_q q}{2^*} \right) C_{N,q}^q a^{(1-\gamma_q)q} t^{\gamma_q q}, \qquad t >0.
\]  
The function $\vartheta$ has a global minimum at negative level
\[
\begin{split}
\bar \vartheta = - \left(\mu a^{(1-\gamma_q)q}\right)^{\frac{2}{2-\gamma_q q}} \left( N \gamma_q \right)^\frac{\gamma_q q}{2-\gamma_q q} \left( C_{N,q}^q\frac{2^*-\gamma_q q}{2 \, 2^*}\right)^\frac{2}{2-\gamma_q q} \frac{2-\gamma_q q}{q} <0,
% \cdot \frac{(2^*-\gamma_q q)(2-\gamma_q q) C_{N,q}^q}{2 q\, 2^* } \cdot \left(\frac{N \gamma_q(2^* - \gamma_q q) C_{N,q}^q}{2 \, 2^*} \right)^\frac{\gamma_q q}{2-\gamma_q q} < 0,
 \end{split}
\]
and since $\mu a^{(1-\gamma_q)q}< \alpha \le C''$ by \eqref{hp}, we have that $\bar \vartheta > -\cS^\frac{N}2/N$. Therefore, from \eqref{7121} we infer that
\[
m(a,\mu) \ge  \frac{\cS^{\frac{N}2}}{N}+ \vartheta(|\nabla u|_2^2) \ge  \frac{\cS^{\frac{N}2}}{N}+ \bar \vartheta >0,
\]
in contradiction with the fact that $m(a,\mu)<0$. This means that necessarily $u_n \to \tilde u$ strongly in $H^1$, and $\tilde u \in S$ is a normalized solution to \eqref{stat com} for some $\tilde \lambda<0$. In fact, $\tilde u$ is a ground state, since $E(\tilde u) = \inf_{\cP} E$, and any other normalized solution stays on $\cP$. It remains only to show that any other ground state is a local minimizer for $E$ on $A_{R_0}$. Let $u$ be a critical point of $E|_S$ at level $m(a,\mu)$. Then $u \in \cP$ and $E(u)<0$, and hence by Lemma \ref{lem: fiber sub} we have that $u \in \cP_+ \subset A_{R_0}$; therefore, always by Lemma \ref{lem: fiber sub}, 
\[
E(u) = m(a,\mu) = \inf_{A_{R_0}} E, \quad \text{and} \quad |\nabla u|_2<R_0. \qedhere
\]
\end{proof}

\section{$L^2$-critical perturbation}\label{sec: cr}

In this section we fix $N \ge 3$, $q=\bar p = 2+4/N$, $a,\mu>0$, and we show that Theorem \ref{thm: main} holds with 
\begin{equation}\label{alpha cr}
\alpha(N,q) = \frac{\bar p}{2 C_{N,\bar p}^{\bar p}} = \bar a_N^\frac{4}{N},
\end{equation}
where $\bar a_N$ coincides with the critical mass for the homogeneous $L^2$-critical Schr\"odinger equation\footnote{This is the only value of $a$ for which the problem
\[
-\Delta u= \lambda u + |u|^\frac{4}{N} u \qquad \text{in $\R^N$}, \quad \text{and} \quad \int_{\R^N} u^2 =a^2
\]
has a positive ground state.}, by \cite{Wei}. Throughout the proof, $a$ and $\mu$ satisfying \eqref{hp} with this definition of $\alpha$ will be fixed, and hence we often omit the dependence on these quantities.

The change of the geometry of $E|_{S}$ with respect to the case $q<\bar p$ is enlightened by the following simple lemmas, which are natural counterparts of those available in the Sobolev subcritical context (see \cite[Section 6]{So}). We recall the decomposition $\cP= \cP_+ \cup \cP_0 \cup \cP_-$, see \eqref{split P}.
%\begin{lemma}\label{lem: struct P cr}
%We have $\cP_0 = \emptyset$, and $\cP$ is a smooth manifold of codimension $2$ in $H$.
%\end{lemma}
%\begin{proof}
If $u \in \cP_0$, that is $\Psi_u'(0) = \Psi_u''(0) = 0$, then necessarily $|u|_{2^*}=0$, which is not possible since $u \in S$. Then $\cP_0 = \emptyset$. Using this fact, and arguing as in \cite[Lemma 5.2]{So}, one can easily check that $\cP$ is a smooth manifold of codimension $1$ in $S_a$.
%\end{proof}

\begin{lemma}\label{lem: fiber cr}
For every $u \in S$, there exists a unique $t_u \in \R$ such that $t_u \star u \in \cP$. $t_u$ is the unique critical point of the function $\Psi_u$, and is a strict maximum point at positive level. Moreover:
\begin{itemize}
\item[1)] $\cP= \cP_-$. 
\item[2)] $\Psi_u$ is strictly decreasing and concave on $(t_u, +\infty)$, and $t_u<0$ implies $P(u)<0$.
\item[3)] The map $u \in S \mapsto t_u \in \R$ is of class $C^1$.
\item[4)] If $P(u)<0$, then $t_u<0$.
\end{itemize}
\end{lemma}
\begin{proof}
The simple proof follows from the fact that
\[
\Psi_u(s) = \left( \frac12 |\nabla u|_2^2 - \frac{\mu}{\bar p} |u|_{\bar p}^{\bar p}\right) e^{2s} - \frac{e^{2^* s}}{2^*} |u|_{2^*}^{2^*},
\]
where
\[
\frac12 |\nabla u|_2^2 - \frac{\mu}{\bar p} |u|_{\bar p}^{\bar p} \ge \left(\frac12 - \frac{\mu}{\bar p} C_{N, \bar p}^{\bar p} a^{\frac4N}\right) |\nabla u|_2^2 >0
\]
having assumed $\mu a^\frac{4}{N} < \alpha$. We refer to \cite[Lemma 6.2]{So} for more details.
\end{proof}

\begin{lemma}\label{lem: E su P cr}
It results that $\displaystyle m(a,\mu):= \inf_{u \in \cP} E(u) >0$.\\
Moreover, there exists $k>0$ sufficiently small such that
\[
0<  \sup_{\overline{A_k}} E < m(a,\mu), \quad \text{and} \quad  u \in \overline{A_k} \implies E(u), P(u) >0,
\]
where $A_k= \left\{u \in S: |\nabla u|_2^2 < k\right\}$. 
\end{lemma}

For the proof it is sufficient to argue as in \cite[Lemma 6.3 and Lemma 6.4]{So} choosing $p=2^*$.

From the above lemmas, it is not difficult to recognize that $E|_S$ has a mountain pass geometry. In order to apply Proposition \ref{prop: PS conv} and recover compactness, we need an estimate from above on $m_r(a,\mu):= \inf_{\cP \cap S_r} E$, where we recall that $S_r$ is the subset of the radial functions in $S$.

\begin{lemma}\label{lem: st above cr}
If $\mu a^{\frac{4}{N}} < \alpha(N,\bar p)$ defined by \eqref{alpha cr}, then $\displaystyle m_r(a,\mu) < \displaystyle{\frac{\cS^\frac{N}2}N}$.
\end{lemma}

\begin{proof}
Let $U_\eps$ be defined by
\begin{equation}\label{def bub 2}
U_\eps(x) := \left( \frac{\eps}{\eps^2 + |x|^2} \right)^\frac{N-2}2
\end{equation}
(up to a scalar factor, $U_\eps$ is the bubble centered in the origin, with concentration parameter $\eps>0$, defined in \eqref{def bubble}). Let also $\varphi \in C^\infty_c(\R^N)$ be a radial cut-off function with $\varphi \equiv 1$ in $B_1$, $\varphi \equiv 0$ in $B_2^c$, and $\varphi$ radially decreasing. We define
\begin{equation}\label{def u e v}
u_\eps(x):= \varphi(x) U_\eps(x), \quad \text{and} \quad v_\eps(x):= a\frac{u_\eps(x)}{|u_\eps|_2}.
\end{equation}
Notice that $u_\eps \in C^\infty_c(\R^N)$, and $v_\eps \in S_r$. Therefore, by Lemma \ref{lem: fiber cr} 
\[
m_r(a,\mu) = \inf_{\cP_{a,\mu} \cap S_r} E_\mu \le E_\mu(t_{v_\eps,\mu} \star v_\eps) = \max_{s \in \R} E_\mu( s \star v_\eps) \qquad \forall \eps>0,
\]
and hence in what follows we focus on an upper estimate of $\max_{s \in \R} E_\mu( s \star v_\eps) = \max_{s \in \R} \Psi_{v_\eps}^\mu(s)$. We split the argument in several steps for the reader's convenience.

\medskip

\noindent \textbf{Step 1)} \emph{Estimate on $\sup_{\R} \Psi_{v_\eps}^0$}. If $\mu=0$ it is possible to make explicit computations using \eqref{def psi 0}, as in Section \ref{sec: pre}: we have that $\Psi_{v_\eps}^0$ has a unique critical point $t_{\eps,0}$, which is a strict maximum point, given by
\begin{equation}\label{def t_0}
e^{t_{\eps,0}}= \left( \frac{|\nabla v_\eps|_2^2}{|v_\eps|_{2^*}^{2^*}} \right)^\frac{1}{2^*-2}. 
\end{equation}
The maximum level is
\begin{equation}\label{stima 0}
\begin{split}
\Psi_{v_\eps}^0(t_{\eps,0}) %&= \frac12 \left( \frac{|\nabla v_\eps|_2^2}{|v_\eps|_{2^*}^{2^*}}\right)^\frac{2}{2^*-2} |\nabla v_\eps|_2^2 - \frac{1}{2^*} \left( \frac{|\nabla v_\eps|_2^2}{|v_\eps|_{2^*}^{2^*}}\right)^\frac{2^*}{2^*-2} |v_\eps|_{2^*}^{2^*} \\
& = \frac1N \left( \frac{|\nabla v_\eps|_2^{2}}{|v_\eps|_{2^*}^{ 2}}\right)^\frac{2^*}{2^*-2} =  \frac1N \left( \frac{|\nabla u_\eps|_2^{2}}{|u_\eps|_{2^*}^{2}}\right)^\frac{N}2   = \frac1N \left( \frac{K_{1} + O(\eps^{N-2})}{K_{2} + O(\eps^{N})}\right)^\frac{N}2 \\
& = \frac{1}{N} \left( \frac{K_{1}}{K_{2}} + O(\eps^{N-2}) \right)^\frac{N}{2}= \frac{\cS^\frac{N}{2}}N + O(\eps^{N-2})
\end{split}
\end{equation}
as $\eps \to 0$, where we used the asymptotic estimates collected in Lemma \ref{lem: app} for $N \ge 4$, and the fact that $K_1/K_2 = \cS$. If $N=3$, the same estimate holds eventually, using $|u_\eps|_{2^*}^{2} = K_2 + O(\eps^2)$ instead of $|u_\eps|_{2^*}^{2} = K_2 + O(\eps^N)$.
\medskip

\noindent \textbf{Step 2)} \emph{Estimate on $t_{\eps,\mu}$}. We denote by $t_{\eps,\mu}:= t_{v_\eps,\mu}$, the unique maximum point of $\Psi_{v_\eps}^\mu$ (see Lemma \ref{lem: fiber cr}). By definition $P_\mu(t_{\eps,\mu} \star v_{\eps}) = 0$, and this implies that 
%\begin{equation*}%\label{char t eps}
%\begin{split}
%|v_\eps|_{2^*}^{2^*} e^{2^* t_{\eps,\mu}} = |\nabla v_\eps|_2^2 e^{2 t_{\eps,\mu}} - \mu \gamma_{\bar p} |v_\eps|_{\bar p}^{\bar p} e^{2 t_{\eps,\mu}}  \le |\nabla v_\eps|_2^2 e^{2 t_{\eps,\mu}},   %\\
%%&2 |\nabla v_\eps|_2^2 e^{2 t_{\eps,\mu}} < \mu q \gamma_q^2  |v_\eps|_q^q e^{\gamma_q q t_{\eps,\mu}} +  2^*|v_\eps|_{2^*}^{2^*} e^{2^* t_{\eps,\mu}}.
%\end{split}
%\end{equation*}
%whence it follows that 
%\begin{equation}\label{char t eps}
%e^{(2^*-2) t_{\eps,\mu}} \le \frac{|\nabla v_\eps|_2^2}{|v_\eps|_{2^*}^{2^*}} = e^{(2^*-2) t_{\eps,0}}.
%\end{equation}
%Using again 
\begin{equation*}%\label{char t eps cr}
\begin{split}
e^{(2^*-2) t_{\eps,\mu}} = \frac{|\nabla v_\eps|_2^2}{|v_\eps|_{2^*}^{2^*}} - \frac{2\mu}{\bar p} \cdot \frac{|v_\eps|_{\bar p}^{\bar p}}{|v_\eps|_{2^*}^{2^*}} \ge  \left( 1- \frac{2}{\bar p} C_{N, \bar p}^{\bar p} \mu a^\frac{4}{N}  \right) \frac{|\nabla v_\eps|_2^2}{|v_\eps|_{2^*}^{2^*}}.
\end{split}
\end{equation*}
%Therefore
%
%
%
%
%
%
%
%
%Combining this inequalities, we obtain
%\[
%\left(2-\gamma_q q\right) |\nabla v_\eps|_2^2 e^{2 t_{\eps,\mu}} < \left(2^*- \gamma_q q\right) |v_\eps|_{2^*}^{2^*} e^{2^* t_{\eps,\mu}}.
%\] 
%Notice that both the terms inside the brackets are positive, since $\gamma_q q <2$. Therefore, using the asymptotic estimates in Lemma \ref{lem: app}, we have
%\begin{equation}\label{t eps sub}
%\begin{split}
%e^{(2^*-2) t_{\eps,\mu}} & > \left(\frac{2-\gamma_q q}{2^*-\gamma_q q}\right) \frac{|\nabla v_\eps|_2^2}{|v_\eps|_{2^*}^{2^*}} = C \frac{|u_\eps|_2^{2^*-2}|\nabla u_\eps|_2^2}{a^{2^*-2}|u_\eps|_{2^*}^{2^*}} \ge \frac{C}{a^{2^*-2}} |u_\eps|_2^{2^*-2} %\ge \begin{cases}  \frac{C}{a^{2^*-2}} \eps^{2^*-2} & \text{if $N \ge 5$} \\ \frac{C}{a^{2^*-2}} \eps^{2^*-2} |\log \eps|^\frac{2^*-2}2 & \text{if $N = 4$}  \\
%%\frac{C}{a^{2^*-2}} \eps^\frac{2^*-2}2 & \text{if $N = 3$}
%%\end{cases}
%\end{split}
%\end{equation}

\medskip

\noindent \textbf{Step 3)} \emph{Estimate on $\sup_\R \Psi^\mu_{v_\eps}$}. By steps 1 and 2  
\begin{equation}\label{stima finale}
\begin{split}
\sup_{\R} \Psi^\mu_{v_\eps} &=  \Psi^\mu_{v_\eps}(t_{\eps,\mu}) = \Psi^0_{v_\eps}(t_{\eps,\mu}) - \frac{\mu}{{\bar p}} e^{2 t_{\eps,\mu}} |v_\eps|_{\bar p}^{\bar p} \\
& \le \sup_\R \Psi^0_{v_\eps} - \frac{\mu}{\bar p}
\left( 1- \frac{2}{\bar p} C_{N, \bar p}^{\bar p} \mu a^\frac{4}{N}  \right)^\frac{2}{2^*-2} \left(\frac{|\nabla v_\eps|_2^2}{|v_\eps|_{2^*}^{2^*}}\right)^\frac{2}{2^*-2} |v_\eps|_{\bar p}^{\bar p} \\
& \le \frac{1}{N} \cS^\frac{N}2 + O(\eps^{N-2}) - \frac{1}{\bar p} \left( 1- \frac{2}{\bar p} C_{N, \bar p}^{\bar p} \mu a^\frac{4}{N}  \right)^\frac{2}{2^*-2} \mu a^\frac4N \frac{|u_\eps|_{\bar p}^{\bar p} |\nabla u_\eps|_2^\frac4{2^*-2}}{|u_\eps|_2^{\frac{4}{N}}  |u_\eps|_{2^*}^\frac{2 2^*}{2^*-2}} \\
&\le \frac{1}{N} \cS^\frac{N}2 + O(\eps^{N-2}) - C_{N,a,\mu} \frac{|u_\eps|_{\bar p}^{\bar p}}{|u_\eps|_2^{\frac{4}{N}}},
\end{split}
\end{equation}
where $C_{N,a,\mu}>0$ is a positive constant independent of $\eps$, and we used Lemma \ref{lem: app}. Always by Lemma \ref{lem: app}, we have
\[
\frac{|u_\eps|_{\bar p}^{\bar p}}{|u_\eps|_2^{\frac{4}{N}}} \ge  \begin{cases}  C \eps^{N-\frac{N-2}2 \bar p - \frac4N} = C & \text{if $N \ge 5$} \\
C \eps^{4-\bar p -1} |\log \eps|^{- \frac12} = C |\log \eps|^{-\frac12}  & \text{if $N=4$}  \\
C \eps^{3-\frac{\bar p}2  - \frac{2}3} = C \eps^\frac23 & \text{if $N=3$}.
%C \eps^{3-\frac{q}2 - (1-\gamma_q)\frac{q}2} = C \eps^{\frac{6-q}4} & \text{if $N=3$ and $q \in (3,10/3)$} \\
%C \eps^{\frac{3}{4}} |\log \eps|^3  & \text{if $N=3$ and $q = 3$} \\
%C \eps^{\frac{q}2 - (1-\gamma_q) \frac{q}2} =  C \eps^{\frac{3(q-2)}4}    & \text{if $N=3$ and $q \in (2,3)$}. 
\end{cases} 
\]
In particular, any term of order $\eps^{N-2}$ is negligible with respect to this ratio for $\eps$ small, and hence coming back to \eqref{stima finale} we deduce that
\[
\sup_{\R} \Psi^\mu_{v_\eps} < \frac{\cS^\frac{N}2}{N}
\]
for any $\eps>0$ small enough, which in turn gives the thesis of the lemma. 
%
%
%
%
%
%
%
%
%
%
%By step 1 and \eqref{t eps sub} 
%\begin{equation}\label{stima finale}
%\begin{split}
%\sup_{\R} \Psi^\mu_{v_\eps} &=  \Psi^\mu_{v_\eps}(t_{\eps,\mu}) = \Psi^0_{v_\eps}(t_{\eps,\mu}) - \frac{\mu}{q} e^{\gamma_q q t_{\eps,\mu}} |v_\eps|_q^q \\
%& \le \sup_\R \Psi^0_{v_\eps} - C \mu a^{(1-\gamma_q)q} \frac{|u_\eps|_q^q}{|u_\eps|_2^{(1-\gamma_q)q}} \\
%& \le \frac{\cS^\frac{N}2}N + O(\eps^{N-2}) - C \mu a^{(1-\gamma_q)q} \frac{|u_\eps|_q^q}{|u_\eps|_2^{(1-\gamma_q)q}}.
%\end{split}
%\end{equation}
%Moreover, by Lemma \ref{lem: app}
%\[
%\frac{|u_\eps|_q^q}{|u_\eps|_2^{(1-\gamma_q)q}} \ge  \begin{cases}  C \eps^{N-\frac{N-2}2 q - (1-\gamma_q) q} = C & \text{if $N \ge 5$} \\
%C \eps^{N-\frac{N-2}2 q - (1-\gamma_q) q} |\log \eps|^{(\gamma_q -1)q} = C |\log \eps|^{(\gamma_q -1)q}  & \text{if $N=4$}  \\
%C \eps^{3-\frac{q}2 - (1-\gamma_q)\frac{q}2} = C \eps^{\frac{6-q}4} & \text{if $N=3$ and $q \in (3,10/3)$} \\
%C \eps^{\frac{3}{4}} |\log \eps|^3  & \text{if $N=3$ and $q = 3$} \\
%C \eps^{\frac{q}2 - (1-\gamma_q) \frac{q}2} =  C \eps^{\frac{3(q-2)}4}    & \text{if $N=3$ and $q \in (2,3)$}. 
%\end{cases} 
%\]
%In particular, any term of order $\eps^{N-2}$ is negligible with respect to this ratio for $\eps$ small, and hence coming back to \eqref{stima finale} we deduce that
%\[
%\sup_{\R} \Psi^\mu_{v_\eps} < \frac{\cS^\frac{N}2}{N}
%\]
%for any $\eps>0$ small enough, which in turn gives the thesis of the lemma. 
\end{proof}

Now we need a technical result. We recall that $T_u S$ denotes the tangent space to $S$ in $u$.
\begin{lemma}{\cite[Lemma 3.6]{BaSo2}}\label{lem: tg}
For $u\in S_a$ and $s\in\R$ the map
\[
T_{u} S \to T_{s \star u} S, \quad \varphi \mapsto s \star \varphi
\]
is a linear isomorphism, with inverse $\psi\mapsto (-s) \star \psi$. 
\end{lemma}
%\begin{proof}
%We refer to \cite[Lemma 3.6]{BaSo2}.
%\end{proof}

We are finally ready for the:

\begin{proof}[Proof of Theorem \ref{thm: main} for $q = \bar p$]
Let $k>0$ be defined by Lemma \ref{lem: E su P cr}. We follow the strategy firstly introduced in \cite{Jea}, considering the augmented functional $\tilde E: \R \times H^1 \to \R$ defined by 
\begin{equation}\label{def aug}
\tilde E(s,u) := E(s \star u) = \frac{e^{2s}}{2} \int_{\R^N}|\nabla u|^2 - \mu \frac{e^{\gamma_{\bar p} \bar p s}}{\bar p} \int_{\R^N} |u|^q -  \frac{e^{2^*s}}{2^*} \int_{\R^N} |u|^{2^*},
\end{equation}
and look at the restriction $\tilde E|_{\R \times S}$. Notice that $\tilde E$ is of class $C^1$. Moreover, since $\tilde E$ is invariant under rotations applied to $u$, a Palais-Smale sequence for $\tilde E|_{\R \times S_r}$ is a Palais-Smale sequence for $E|_{\R \times S}$. Denoting by $E^c$ the closed sublevel set $\{u \in S: E(u) \le c\}$, we introduce the minimax class
\begin{equation}\label{def gamma}
\Gamma:= \left\{ \gamma = (\alpha,\beta) \in C([0,1], \R \times S_{r}): \ \gamma(0) \in (0,\overline{A_k}), \ \gamma(1) \in (0,E^{0})\right\},
\end{equation}
with associated minimax level
\[
\sigma(a,\mu):= \inf_{\gamma \in \Gamma} \max_{(s,u) \in \gamma([0,1])} \tilde E(s,u).
\]
Let $u \in S_r$. Since $|\nabla (s \star u)|_2 \to 0^+$ as $ s \to -\infty$, and $E(s \star u) \to -\infty$ as $s \to +\infty$, there exist $s_0 \ll -1$ and $s_1 \gg 1$ such that 
\begin{equation}\label{gam u}
\gamma_u: \tau \in [0,1] \mapsto (0,((1-\tau) s_0 + \tau s_1) \star u) \in \R \times S_r
\end{equation}
is a path in $\Gamma$ (the continuity follows from \eqref{cont star}). Then $\sigma(a,\mu)$ is a real number.

Now, for any $\gamma =(\alpha,\beta) \in \Gamma$, let us consider the function
\[
P_\gamma:\tau \in [0,1] \mapsto P(\alpha(\tau) \star \beta(\tau)) \in \R.
\]
We have $P_\gamma(0) = P(\beta(0)) >0$, by Lemma \ref{lem: E su P cr}, and we claim that $P_\gamma(1) =P(\beta(1))<0$: indeed, since $\Psi_{\beta(1)}(s)>0$ for every $s \in (-\infty,t_{\beta(1)}]$, and $\Psi_{\beta(1)}(0) = E(\beta(1)) \le 0$, we have that $t_{\beta(1)} <0$. Then the claim follows by Lemma \ref{lem: fiber cr}. Moreover, the map $\tau \mapsto \alpha(\tau) \star \beta(\tau)$ is continuous from $[0,1]$ to $H^1$ by \eqref{cont star}, and hence we deduce that there exists $\tau_\gamma \in (0,1)$ such that $P_\gamma(\tau_\gamma) = 0$, namely $\alpha(\tau_\gamma) \star \beta(\tau_\gamma) \in \cP$; this implies that 
\[
\max_{\gamma([0,1])} \tilde E \ge \tilde E(\gamma(\tau_\gamma)) =E( \alpha(\tau_\gamma) \star \beta(\tau_\gamma)) \ge \inf_{\cP\cap S_r} E = m_r(a,\mu),
\]
and consequently $\sigma(a,\mu) \ge m_r(a,\mu)$. On the other hand, if $u \in \cP_- \cap S_r$, then $\gamma_u$ defined in \eqref{gam u} is a path in $\Gamma$ with 
\[
E(u) = \max_{\gamma_u([0,1])} \tilde E \ge \sigma(a,\mu),
\]
whence the reverse inequality $m_r(a,\mu) \ge \sigma(a,\mu)$ follows. Combining this with Lemmas \ref{lem: E su P cr}, we infer that
\begin{equation}\label{stima infmax}
\sigma(a,\mu) = m_r(a,\mu) > \sup_{(\overline{A_k}  \cup E^{0}) \cap S_r} E = \sup_{((0,\overline{A_k})  \cup (0,E^{0})) \cap (\R \times S_r)} \tilde E.
\end{equation}
Using the terminology in \cite[Section 5]{Gho}, this means that $\{\gamma([0,1]): \ \gamma \in \Gamma\}$ is a homotopy stable family of compact subsets of $\R \times S_r$ with extended closed boundary $(0,\overline{A_k}) \cup (0,E^{0})$, and that the superlevel set $\{\tilde E \ge \sigma(a,\mu)\}$ is a dual set for $\Gamma$, in the sense that assumptions (F'1) and (F'2) in \cite[Theorem 5.2]{Gho} are satisfied. Therefore, taking any minimizing sequence $\{\gamma_n=(\alpha_n,\beta_n)\} \subset \Gamma_n$ for $\sigma(a,\mu)$ with the property that $\alpha_n \equiv 0$ and $\beta_n(\tau) \ge 0$ a.e. in $\R^N$ for every $\tau \in [0,1]$\footnote{Notice that, if $\{\gamma_n=(\alpha_n,\beta_n)\} \subset \Gamma$ is a minimizing sequence, then also $\{(0,\alpha_n \star |\beta_n|)\}$ has the same property.}, there exists a Palais-Smale sequence $\{(s_n,w_n)\} \subset \R \times S_r$ for $\tilde E|_{\R \times S_r}$ at level $\sigma(a,\mu)$, that is 
\begin{equation}\label{1892}
\pa_s \tilde E(s_n,w_n) \to 0 \quad \text{and} \quad \|\pa_u \tilde E(s_n,w_n)\|_{(T_{w_n} S_{r})^*} \to 0 \quad \text{as $n \to \infty$},
\end{equation}
with the additional property that
\begin{equation}\label{1893}
|s_n|  + \dist_{H^1}(w_n, \beta_n([0,1])) \to 0 \qquad \text{as $n \to \infty$}.
\end{equation}
By \eqref{def aug}, the first condition in \eqref{1892} reads $P(s_n \star w_n) \to 0$, while the second condition gives
%\[
%e^{2s_n} \int_{\R^N} \nabla w_n \cdot \nabla \varphi - \mu e^{\gamma_q q s_n} \int_{\R^N} |w_n|^{q-2}w_n \varphi - e^{2^* s_n} \int_{\R^N} |w_n|^{2^*-2}w_n \varphi = o(1) \|\varphi\|
%\]
%for every $\varphi \in T_{w_n} S_r$, with $o(1) \to 0$ as $n \to \infty$. %\footnote{Here we used the fact that the functional $\tilde E$ is invariant under rotations of its second variable $u$, and hence a PS sequence of $\tilde E|_{\R \times S_{a,r}}$ is a PS sequence for $\tilde E|_{\R \times S_a}$.}. 
%Since $\{s_n\}$ is bounded from above and from below, due to \eqref{1893}, this is equivalent to 
\begin{equation}\label{9105}
dE(s_n \star w_n)[s_n \star \varphi] = o(1) \|\varphi\| = o(1) \|s_n \star \varphi\| \quad \text{as $n \to \infty$, for every $\varphi \in T_{w_n} S_r$},
\end{equation}
for every $\varphi \in T_{w_n} S_r$, with $o(1) \to 0$ as $n \to \infty$; in the last equality, we used that $\{s_n\}$ is bounded from above and from below, due to \eqref{1893}.
Let then $u_n:= s_n \star w_n$. By Lemma \ref{lem: tg}, equation \eqref{9105} establishes that $\{u_n\} \subset S_{r}$ is a Palais-Smale sequence for $E|_{S_{r}}$ (thus a PS sequence for $E|_S$, since the problem is invariant under rotations) at level $\sigma(a,\mu) =m_r(a,\mu)$, with $P(u_n) \to 0$. By Lemmas \ref{lem: E su P cr} and \ref{lem: st above cr}, $m_r(a,\mu) \in (0, \cS^\frac{N}2/N)$, and hence one of the two alternatives in Proposition \ref{prop: PS conv} occurs. Let us assume that alternative ($i$) takes place; then there exists $\tilde u \in H^1$ such that $u_n \weak \tilde u$ weakly in $H^1$, but not strongly, and 
\begin{equation}\label{E weak lim cr}
E(\tilde u) \le m_r(a,\mu) - \frac{\cS^\frac{N}2}N < 0.
\end{equation}
Moreover, the limit $\tilde u \not \equiv 0$ solves \eqref{stat com} for some $\lambda<0$, and hence by the Pohozaev identity $P(\tilde u) = 0$. But then $E(\tilde u) =  |\tilde u|_{2^*}^{2^*}/N>0$, a contradiction with \eqref{E weak lim cr}. This shows that necessarily alternative ($ii$) in Proposition \ref{prop: PS conv} holds, namely $u_n \to \tilde u$ strongly in $H^1$, and $\tilde u$ is a real valued radial normalized solution to \eqref{stat com}-\eqref{norm} for some $\tilde \lambda<0$, with energy $m_r(a,\mu)$. Recalling that $\beta_n(\tau) \ge 0$ a.e. in $\R^N$ for every $\tau$, condition \eqref{1893} and the convergence imply that $\tilde u$ is also non-negative, and hence positive by the maximum principle. It remains only to prove that $\tilde u$ is a ground state. Since any normalized solutions stays on $\cP$, and $E(\tilde u) = m_r(a,\mu) = \inf_{\cP \cap S_r} E$, it is sufficient to check that $ \inf_{\cP \cap S_r} E = \inf_{\cP} E = m(a,\mu)$. Suppose by contradiction that this is not the case, that is that there exists $u \in \cP \setminus S_r$ with $E(u) < \inf_{\cP \cap S_r} E$. Then we let $v:= |u|^*$, the symmetric decreasing rearrangement of the modulus of $u$, which lies in $S_r$. By standard properties $|\nabla v|_2 \le |\nabla u|_2$, $E(v) \le E(u)$, and $P(v) \le P(u) = 0$. If $P(v) = 0$ we have a contradiction, and hence we can assume that $P(v) <0$. In this case, from Lemma \ref{lem: fiber cr} we know that $t_v<0$. But then we obtain again a contradiction in the following way:
\begin{align*}
E(u) &< E(t_v \star v) = \frac{e^{2^* t_v}}N |v|_{2^*}^{2^*} = \frac{e^{2^* t_v}}N |u|_{2^*}^{2^*} = e^{2^* t_v} E(u) < E(u),
% \left( \frac1N |\nabla v|_2^2 - \frac{2\mu}{N\bar p} |v|_{\bar p}^{\bar p}\right) \\
%& \le e^{2 t_v} \left( \frac1N |\nabla u|_2^2 - \frac{2\mu}{N\bar p} |u|_{\bar p}^{\bar p}\right)  = e^{2 t_v} E(u) < E(u),
\end{align*}
where we used the fact that $t_v \star v$ and $u$ lies in $\cP$. This proves that $m_r(a,\mu)=m(a,\mu)$, and hence $\tilde u$ is a ground state. 
\end{proof}

\section{$L^2$-supercritical perturbation}\label{sec: sup}

In this section we prove Theorem \ref{thm: main} for $2+4/N < q <2^*$. To be precise, we show that the theorem holds with 
\begin{equation}\label{alpha sup}
\alpha(N,q) = +\infty \quad \text{if $N=3,4$, and } \quad \alpha(N,q) = \frac{\cS^{\frac{N}{4}(1-\gamma_q)q}}{\gamma_q} \quad \text{if $N \ge 5$}.
%\alpha(N,q) = \begin{cases} +\infty & \text{if $N=3,4$} \\ \dfrac{\cS^\frac{2^*-\gamma_q q}{2^*-2}}{\gamma_q C_{N,q}^q} & \text{if $N \ge 5$}
%\end{cases}
\end{equation}
Throughout the proof, $a$ and $\mu$ satisfying \eqref{hp} with this definition of $\alpha$ will be fixed, and hence we often omit the dependence on these quantities. We consider once again the Pohozaev manifold $\cP$, defined in \eqref{def P}, and the decomposition $\cP= \cP_+\cup \cP_0 \cup \cP_-$, see \eqref{split P}. If there exists $u \in \cP_0$, then combining $\Psi_u'(0) = 0$ and $\Psi_u''(0) = 0$ we deduce that
\[
\underbrace{(2-q \gamma_q)}_{<0 \text{ since $q> \bar p$}} \mu \gamma_q |u|_q^q = \underbrace{(2^*-2)}_{>0} |u|_{2^*}^{2^*},
\]
whence necessarily $u \equiv 0$, in contradiction with the fact that $u \in S$. This shows that $\cP_0= \emptyset$. At this point we can prove that  $\cP$ is a manifold in a standard way, see \cite[Lemma 5.2]{So}.

\begin{lemma}\label{lem: fiber sup}
For every $u \in S$, there exists a unique $t_u \in \R$ such that $t_u \star u \in \cP$. $t_u$ is the unique critical point of the function $\Psi_u$, and is a strict maximum point at positive level. Moreover:
\begin{itemize}
\item[1)] $\cP= \cP_-$. 
\item[2)] $\Psi_u$ is strictly decreasing and concave on $(t_u, +\infty)$, and $t_u<0$ implies $P(u)<0$.
\item[3)] The map $u \in S \mapsto t_u \in \R$ is of class $C^1$.
\item[4)] If $P(u)<0$, then $t_u<0$.
\end{itemize}
\end{lemma}

%
%
%
%\begin{lemma}
%For every $u \in S_{a}$, the function $\Psi^\mu_u$ has a unique critical point $s_u$, which is a strict maximum at positive level. Moreover, we have $s \star u \in \cP$ if and only if $s=s_u$.
%\end{lemma}

\begin{proof}
Clearly $\Psi_u(s) \to 0^+$ as $s \to -\infty$, and $\Psi_u(s) \to -\infty$ as $s \to +\infty$, for every $u \in S_r$. Therefore, $\Psi_u$ has a global maximum point $t_u$ at positive level. To show that this is the unique critical point of $\Psi_u$, we observe that $\Psi_u'(s) = 0$ if and only if
\begin{equation}\label{19112}
\mu \gamma_q |u|_q^q e^{(\gamma_q q-2)s} +  |u|_{2^*}^{2^*} e^{(2^*-2) s} = |\nabla u|_2^2.
\end{equation}
The left hand side, as function of $s \in \R$, is positive, continuous, monotone increasing, with limits $0^+$ and $+\infty$ as $s \to -\infty$ and $s \to +\infty$, respectively. Therefore, equation \eqref{19112} has exactly one solution. By maximality $\Psi_u''(t_u) \le 0$, and since $t_u \star u \in \cP$ and $\cP_0 = \emptyset$, we deduce that $t_u \star u \in \cP_-$. In particular, this and Proposition \ref{prop: psi P} imply that $\cP= \cP_-$.

The same argument used to prove the existence and uniqueness of a critical point can also be used on $\Psi_u''$ to check that $\Psi_u$ has exactly one inflection point. From this, point (2) in the thesis follows. 

Regarding point (3), as in \cite[Lemma 5.3]{So} we apply the implicit function theorem: we let $\Phi(s,u) = \Psi_u'(s)$, and observe that $\Phi$ is of class $C^1$ in the two variables $(s,u) \in \R \times H^1(\R^N)$, $\Phi(t_u,u) = 0$, and $\pa_s\Phi(t_u,u) = \Psi_u''(t_u) <0$. Therefore, $u \in H^1 \mapsto t_u$ is of class $C^1$.

Finally, for point (4) we observe that $\Psi_u'(s) <0$ if and only if $s>t_u$. Therefore, if $P(u) = \Psi_u'(0)<0$, then $t_u<0$.
\end{proof}

\begin{lemma}
It results that $\displaystyle m(a,\mu):= \inf_{u \in \cP} E(u) >0$.
\end{lemma}

\begin{proof}
If $P(u)=0$, then by the Gagliardo-Nirenberg and the Sobolev inequalities 
\[
|\nabla u|_2^2 = \mu \gamma_q |u|_q^q + |u|_{2^*}^{2^*} \le \gamma_q C_{N,q}^q \mu a^{(1-\gamma_q) q} |\nabla u|_2^{\gamma_q q} + \cS^{-\frac{2^*}2} |\nabla u|_2^{2^*},
\]
whence dividing by $|\nabla u|_2^2$ (this is possible since $u \in S$, and hence $|\nabla u|_2 \neq 0$) we deduce that
\[
\mu \gamma_q C_{N,q}^q a^{(1-\gamma_q) q} |\nabla u|_2^{\gamma_q q-2} + \cS^{-\frac{2^*}2} |\nabla u|_2^{2^*-2} \ge 1, \quad \forall u \in \cP.
\]
This implies that $\inf_{u \in \cP} |\nabla u|_2 >0$, and hence, by definition of $\cP$,
\[
\inf_{ u \in \cP}  \left( |u|_q^q + |u|_{2^*}^{2^*}\right)>0,
\]
which finally gives
\[
\inf_{u \in \cP} E(u) = \inf_{u \in \cP} \left[\frac{\mu}q\left(\frac{\gamma_q q}2-1\right) |u|_q^q + \frac1N  |u|_{2^*}^{2^*}\right] >0. \qedhere
\]
\end{proof}

\begin{lemma}\label{lem: stima sup A mu < 0}
There exists $k>0$ sufficiently small such that
\[
0< \sup_{\overline{A_k}} E < m(a,\mu) \quad \text{and} \quad u \in \overline{A_k} \implies E(u), P(u) >0,
\]
where $A_k:= \left\{u \in S: |\nabla u|_2^2 < k\right\}$. 
\end{lemma}

%\begin{lemma}\label{lem: stima sup A}
%There exists $K>0$ such that
%\[
%0< \inf_{u \in A} < \sup_{u \in A} E(u) < c,
%\]
%where
%\begin{align*}
%&A:= \left\{u \in S_r: |\nabla u|_2^2 \le K\right\} 
%\end{align*}
%\end{lemma}

\begin{proof}
By the Gagliardo-Nirenberg and the Sobolev inequalities
\[
\begin{split}
E(u) & \ge \frac12 |\nabla u|_2^2 -  \frac{1}{q} C_{N,q}^q \mu a^{(1-\gamma_q)q}|\nabla u|_2^{\gamma_q q} - \frac{1}{2^*} \cS^{-\frac{2^*}2}|\nabla u|_2^{2^*}  >0, \\
P(u) &\ge  |\nabla u|_2^2 - \gamma_q C_{N,q}^q \mu a^{(1-\gamma_q)q}|\nabla u|_2^{\gamma_q q} - \cS^{-\frac{2^*}2}|\nabla u|_2^{2^*} > 0,
\end{split}
\]
if $u \in \overline{A_k}$ with $k$ small enough, since $\gamma_q q>2$. If necessary replacing $k$ with a smaller quantity, we also have $E(u) \le |\nabla u|_2^2/2 < m(a,\mu)$ for every $u \in \overline{A_k}$.
\end{proof}

As in the previous section, the following estimate will play a crucial role in the proof of existence of a ground state. Let $m_r(a,\mu):= \inf_{\cP \cap S_r} E$, where $S_r=S \cap H^1_{\rad}$.

\begin{lemma}\label{lem: st above sub}
If $\mu a^{(1-\gamma_q)q} < \alpha(N,q)$ defined by \eqref{alpha sup}, then $\displaystyle m_r(a,\mu) < \displaystyle{\frac{\cS^\frac{N}2}N}$.
\end{lemma}

\begin{proof}
The structure of the proof is similar to the one of Lemma \ref{lem: st above cr}, but we have to face some complications, since therein we took advantage of the fact that $\gamma_{\bar p} \bar p=2$ in order to make direct computations in several steps. We define $u_\eps \in C^\infty_c(\R^N)$ and $v_\eps \in S_r$ as in Lemma \ref{lem: st above cr} and, observing that 
\[
m_r(a,\mu) = \inf_{\cP \cap S_r} E_\mu \le E_\mu(t_{v_\eps,\mu} \star v_\eps) = \max_{s \in \R} E_\mu( s \star v_\eps),
\]
we focus on an upper estimate of $\max_{s \in \R} E_\mu( s \star v_\eps) = \max_{s \in \R} \Psi_{v_\eps}^\mu(s)$. It is convenient to recall that
$\Psi_{v_\eps}^0$ has a unique critical point $t_{\eps,0}$, which is a strict maximum point, given by \eqref{def t_0}; the maximum level is estimated in \eqref{stima 0} as
\[
\Psi_{v_\eps}^0(t_{\eps,0}) = \frac{\cS^\frac{N}2}{N} + O(\eps^{N-2}).
\]

\noindent \textbf{Step 1)} \emph{Estimate on $t_{\eps,\mu}$}. We denote by $t_{\eps,\mu}:= t_{v_\eps,\mu}$, the unique maximum point of $\Psi_{v_\eps}^\mu$ (see Lemma \ref{lem: fiber sup}). By definition $P_\mu(t_{\eps,\mu} \star v_{\eps}) = 0$, and hence
\begin{equation}
\begin{split}
|v_\eps|_{2^*}^{2^*} e^{2^* t_{\eps,\mu}} = |\nabla v_\eps|_2^2 e^{2 t_{\eps,\mu}} - \mu \gamma_{q} |v_\eps|_{q}^{q} e^{\gamma_q q t_{\eps,\mu}}  \le |\nabla v_\eps|_2^2 e^{2 t_{\eps,\mu}}, 
\end{split}
\end{equation}
whence it follows that 
\begin{equation}\label{t eps < t 0}
e^{t_{\eps,\mu}} \le \left(\frac{|\nabla v_\eps|_2^2}{|v_\eps|_{2^*}^{2^*}}\right)^\frac{1}{2^*-2}.
\end{equation}
By \eqref{t eps < t 0}, and using the fact that $\gamma_q q >2$, 
\begin{equation}\label{t eps > 1}
\begin{split}
e^{(2^*-2) t_{\eps,\mu}} &= \frac{|\nabla v_\eps|_2^2}{|v_\eps|_{2^*}^{2^*}} - \mu \gamma_q \frac{|v_\eps|_{q}^{q}}{|v_\eps|_{2^*}^{2^*}} e^{(\gamma_q q-2) t_{\eps,\mu}} \ge \frac{|\nabla v_\eps|_2^2}{|v_\eps|_{2^*}^{2^*}} - \mu \gamma_q \frac{|v_\eps|_{q}^{q}}{|v_\eps|_{2^*}^{2^*}} \left(\frac{|\nabla v_\eps|_2^2}{|v_\eps|_{2^*}^{2^*}}\right)^{\frac{\gamma_q q-2}{2^*-2}}\\
& = \frac{|u_\eps|_2^{2^*-2}}{a^{2^*-2}} \frac{|\nabla u_\eps|_2^2}{|u_\eps|_{2^*}^{2^*}} -\mu \gamma_q  \frac{|u_\eps|_2^{2^*-q}}{a^{2^*-q}} \frac{|u_\eps|_q^q}{|u_\eps|_{2^*}^{2^*}} \left(  \frac{|u_\eps|_2^{2^*-2}}{a^{2^*-2}} \frac{|\nabla u_\eps|_2^2}{|u_\eps|_{2^*}^{2^*}} \right)^\frac{\gamma_q q -2}{2^* -2}     \\
& = \frac{|u_\eps|_2^{2^*-2}}{a^{2^*-2}} \frac{\left(|\nabla u_\eps|_2^2\right)^\frac{\gamma_q q -2}{2^*-2}}{|u_\eps|_{2^*}^{2^*}} \left[ \left(|\nabla u_\eps|_2^2\right)^{\frac{2^*-\gamma_q q}{2^*-2}} -  \frac{\gamma_q \mu a^{(1-\gamma_q)q} |u_\eps|_q^q}{\left(|u_\eps|_{2^*}^{2^*}\right)^\frac{\gamma_q q-2}{2^*-2} |u_\eps|_2^{(1-\gamma_q)q}}\right].
\end{split}
\end{equation}
Using the asymptotic estimates in Lemma \ref{lem: app}, we deduce that there exist $C_1, C_2, C_3>0$ (depending on $N$ and $q$ but independent of $\eps<1$, $a$ and $\mu$) such that
\begin{equation}\label{19113}
\begin{split}
\left(|\nabla u_\eps|_2^2\right)^{\frac{2^*-\gamma_q q}{2^*-2}} \ge C_1,  \quad \frac{1}{C_2} \ge  \left(|u_\eps|_{2^*}^{2^*}\right)^\frac{\gamma_q q-2}{2^*-2} \ge C_2,
\end{split}
\end{equation}
and 
\begin{equation}\label{19114}
\frac{|u_\eps|_q^q}{|u_\eps|_2^{(1-\gamma_q)q}} \le  \begin{cases}  C_3 \eps^{N-\frac{N-2}2 q - (1-\gamma_q) q} = C_3 & \text{if $N \ge 5$} \\
C_3 \eps^{N-\frac{N-2}2 q - (1-\gamma_q) q} |\log \eps|^{\frac{(\gamma_q -1)q}2} = C_3 |\log \eps|^{\frac{(\gamma_q -1)q}2}  & \text{if $N=4$}  \\
C_3 \eps^{3-\frac{q}2 - (1-\gamma_q)\frac{q}2} = C_3 \eps^{\frac{6-q}4} & \text{if $N=3$}
\end{cases} 
\end{equation}
(in case $N=3$, we used the fact that $6>q>\bar p = 10/3>3$). Let $N=3,4$. Coming back to \eqref{t eps > 1}, estimates \eqref{19113} and \eqref{19114} implies that for some $o_\eps(1) \to 0$ as $\eps \to 0$ it results that
\[
e^{(2^*-2) t_{\eps,\mu}} \ge \frac{C |u_\eps|_2^{2^*-2}}{a^{2^*-2}} \left[ C_1 - \gamma_q \mu a^{(1-\gamma_q)q} \frac{C_3}{C_2} o_\eps(1)\right] \ge \frac{C}{a^{2^*-2}} |u_\eps|_2^{2^*-2},
\]
provided that $\eps \in (0, \eps_0)$ with $\eps_0$ sufficiently small. Therefore, if $N=3,4$ we proved that for every $a,\mu>0$
\begin{equation}\label{stima t eps sup}
e^{(2^*-2) t_{\eps,\mu}}  \ge \frac{C}{a^{2^*-2}} |u_\eps|_2^{2^*-2}
\end{equation}
for a positive constant $C = C(N,q,\mu,a)>0$, for every $\eps \in (0,\eps_0)$ with $\eps_0>0$ small.

If $N \ge 5$, in view of \eqref{19114} the previous argument only shows that \eqref{stima t eps sup} holds if $\gamma_q \mu a^{(1-\gamma_q) q} < C_1 C_2/C_3$. In order to obtain a more precise quantification, however, it is convenient to come back to \eqref{t eps > 1} and observe that\begin{equation}
\begin{split}
\frac{|u_\eps|_q^q}{\left(|u_\eps|_{2^*}^{2^*}\right)^\frac{\gamma_q q-2}{2^*-2} |u_\eps|_2^{(1-\gamma_q)q}} \le \frac{ |u_\eps|_{2^*}^{2^* \frac{q-2}{2^*-2}} |u_\eps|_2^{2\frac{2^*-q}{2^*-2}}}{|u_\eps|_{2^*}^{2^*\frac{\gamma_q q-2}{2^*-2}} |u_\eps|_2^{(1-\gamma_q)q}} =  |u_\eps|_{2^*}^{2^*\frac{(1-\gamma_q)q}{2^*-2}} =   (|u_\eps|_{2^*}^2)^\frac{2^*-\gamma_q q}{2^*-2}.
\end{split}
\end{equation}
Therefore, \eqref{t eps > 1} gives
\[
e^{(2^*-2) t_{\eps,\mu}} \ge \frac{|u_\eps|_2^{2^*-2}}{a^{2^*-2}} \frac{\left(|\nabla u_\eps|_2^2\right)^\frac{\gamma_q q -2}{2^*-2}}{|u_\eps|_{2^*}^{2^*}} \left[ \left(|\nabla u_\eps|_2^2\right)^{\frac{2^*-\gamma_q q}{2^*-2}} - \gamma_q  \mu a^{(1-\gamma_q)q}  (|u_\eps|_{2^*}^2)^\frac{2^*-\gamma_q q}{2^*-2} \right].
\]
The right hand side is positive provided that
\[
\gamma_q \mu a^{(1-\gamma_q)q} < \left(\frac{|\nabla u_\eps|_2^2}{|u_\eps|_{2^*}^2}\right)^{\frac{2^*-\gamma_q q}{2^*-2}} = \left(\frac{|\nabla u_\eps|_2^2}{|u_\eps|_{2^*}^2}\right)^{\frac{N}{4} (1-\gamma_q)q}= \cS^{\frac{N}{4} (1-\gamma_q)q} + O(\eps^{N-2}),
\]
where the last estimate follows from the definition of $u_\eps$ as in \eqref{stima 0}. Therefore, if \eqref{hp} holds with $\alpha(N,q)$ defined by \eqref{alpha sup}, using again Lemma \ref{lem: app} we conclude that
\begin{equation}\label{stima t eps sup 2}
e^{(2^*-2) t_{\eps,\mu}}  \ge \frac{C}{a^{2^*-2}} |u_\eps|_2^{2^*-2}
\end{equation}
for a positive constant $C = C(N,q,\mu,a)>0$, for every $\eps \in (0,\eps_0)$ with $\eps_0>0$ small.

\medskip

\noindent \textbf{Step 2)} \emph{Estimate on $\sup_\R \Psi^\mu_{v_\eps}$}. By steps 1 and 2  
\begin{equation}\label{stima finale 2}
\begin{split}
\sup_{\R} \Psi^\mu_{v_\eps} &=  \Psi^\mu_{v_\eps}(t_{\eps,\mu}) = \Psi^0_{v_\eps}(t_{\eps,\mu}) - \frac{\mu}{{q}} e^{\gamma_q q t_{\eps,\mu}} |v_\eps|_q^q \\
& \le \sup_\R \Psi^0_{v_\eps} - \frac{ \mu C}{qa^{\gamma_q q} } |u_\eps|_2^{\gamma_q q} \cdot \frac{a^q |u_\eps|_q^q}{|u_\eps|_2^q}  \\
& =    \frac{1}{N} \cS^\frac{N}2 + O(\eps^{N-2}) - C \mu a^{(1-\gamma_q)q}  \frac{|u_\eps|_{q}^{q}}{|u_\eps|_2^{(1-\gamma_q)q}},
\end{split}
\end{equation}
where $C>0$ is independent of $\eps$. Similarly as in \eqref{19114}, we have that
\begin{equation}\label{191142}
\frac{|u_\eps|_q^q}{|u_\eps|_2^{(1-\gamma_q)q}} \ge  \begin{cases}  C_4 \eps^{N-\frac{N-2}2 q - (1-\gamma_q) q} = C_4 & \text{if $N \ge 5$} \\
C_4 \eps^{N-\frac{N-2}2 q - (1-\gamma_q) q} |\log \eps|^{\frac{(\gamma_q -1)q}2} = C_4 |\log \eps|^{\frac{(\gamma_q -1)q}2}  & \text{if $N=4$}  \\
C_4 \eps^{3-\frac{q}2 - (1-\gamma_q)\frac{q}2} = C_4 \eps^{\frac{6-q}4} & \text{if $N=3$}
\end{cases} 
\end{equation}
for a constant $C_4>0$, and hence we finally infer that $\sup_{\R} \Psi^\mu_{v_\eps} < \cS^\frac{N}2/N$ for any $\eps>0$ small enough, which is the desired result. 
\end{proof}

%\begin{lemma}\label{lem: def u_1 u_2}
%There exist $u_1, u_2 \in S_r$ such that
%\[
%\left\{
%\begin{array}{l l l}
%u_1 \in A, &  P_\mu(u_1) >0, & u_1>0 \quad \text{in $\R^N$},   \\E(u_2) <0, & P_\mu(u_2) <0, & u_2>0 \quad \text{in $\R^N$}. \end{array} \right.
%\]
%\end{lemma}
%
%\begin{proof}
%For any positive $u \in S_r$, we consider $s \star u$. If $s_1 \ll -1$, then we have that $|\nabla (s_1 \star u)|_2^2  = e^{2s_1} |\nabla u|_2^2 \le K$, and moreover $(\Psi^\mu_u)'(s_1) >0$. But by definition $(\Psi^\mu_u)'(s_1) = P_\mu(s_1 \star u)$, and hence the choice $u_1 = s_1 \star u$ fulfills $u_1 \in A$ and $P_\mu(u_1) >0$. In a similar way, if $s_2 \gg 1$ we have $E(s_2 \star u) <0$ and $P_\mu(s_2 \star u) = (\Psi^\mu_u)'(s_2) <0$; thus we can take $u_2 = s_2 \star u$.
%\end{proof}
%
%It is now not difficult to proceed with the:

\begin{proof}[Proof of Theorem \ref{thm: main} for $2+4/N<q<2^*$]
We consider the augmented functional $\tilde E$ defined as in \eqref{def aug}, and the minimax class $\gamma$ defined as in \eqref{def gamma}. Proceeding exactly as in the case $q=\bar p$, we obtain a Palais-Smale sequence $\{u_n\} \subset S_r$ for $E|_{S}$ at level $m_r(a,\mu) \in (0,\cS^\frac{N}2/N)$, with the property that $P(u_n) \to 0$, and $u_n^- \to 0$ a.e. in $\R^N$. Then
%
%\begin{equation}\label{19115}
%\dist_{H^1}(u_n, \cP) \to 0 \quad \text{and} \quad \dist_{H^1}(u_n, \gamma_n([0,1])) \to 0.
%\end{equation}
%The first condition in \eqref{19115} gives $P(u_n) \to 0$. Since $\sigma(a,\mu) =m_r(a,\mu) \in (0,\cS^\frac{N}2/N)$, 
one of the two alternatives in Proposition \ref{prop: PS conv} occurs.
%\[
%\Gamma:= \left\{ \gamma \in C([0,1], S_r): \ \gamma(0) \in \overline{A_k}, \ E(\gamma(1))  \le 0\right\},
%\]
%with associated minimax level
%\[
%\sigma(a,\mu):= \inf_{\gamma \in \Gamma} \max_{u \in \gamma([0,1])} E(u).
%\]
%As in the $L^2$-critical case, it is possible to prove that $\sigma(a,\mu) =m_r(a,\mu)$, and that the assumptions of \cite[Theorem 5.2]{Gho} are satisfied. Therefore, for any minimizing sequence $\{\gamma_n\} \subset \Gamma$ for $\sigma(a,\mu)$ with the property that $\gamma_n(\tau) \ge 0$ a.e. in $\R^N$ for every $\tau$, we obtain a Palais-Smale sequence $\{u_n\} \subset S_r$ for $E|_{S}$ at level $\sigma(a,\mu)$, with the property that 
%\begin{equation}\label{19115}
%\dist_{H^1}(u_n, \cP) \to 0 \quad \text{and} \quad \dist_{H^1}(u_n, \gamma_n([0,1])) \to 0.
%\end{equation}
%The first condition in \eqref{19115} gives $P(u_n) \to 0$. Since $\sigma(a,\mu) =m_r(a,\mu) \in (0,\cS^\frac{N}2/N)$, one of the two alternatives in Proposition \ref{prop: PS conv} occurs. 
Let us assume that alternative ($i$) takes place: there exists $\tilde u \in H^1$ such that $u_n \weak \tilde u$ weakly in $H^1$, but not strongly; $\tilde u \not \equiv 0$ solves \eqref{stat com} for some $\lambda<0$, and
\begin{equation}\label{E weak lim sup}
E(\tilde u) \le m_r(a,\mu) - \frac{\cS^\frac{N}2}N < 0.
\end{equation}
However, since $P(\tilde u) = 0$ by the Pohozaev identity and $\gamma_q q>2$, we also have
\[
E(\tilde u) = \frac{\mu}q\left(\frac{\gamma_q q}2-1\right)|\tilde u|_q^q + \frac{1}{N}|\tilde u|_{2^*}^{2^*} > 0,
\]
a contradiction. This shows that necessarily alternative ($ii$) in Proposition \ref{prop: PS conv} holds, namely $u_n \to \tilde u$ strongly in $H^1$, and $\tilde u$ is a real-valued radial normalized solution to \eqref{stat com}-\eqref{norm} for some $\tilde \lambda<0$, with energy $m_r(a,\mu)$. By convergence, $\tilde u$ is also non-negative, and it remains to prove that $\tilde u$ is a ground state. It is not difficult sufficient to check, as in the $L^2$-critical case, that $ \inf_{\cP \cap S_r} E = \inf_{\cP} E = m(a,\mu)$. The thesis follows.
% If by contradiction there exists $u \in \cP \setminus S_r$ with $E(u) < \inf_{\cP \cap S_r} E$, then we consider $v:= |u|^*$, the symmetric decreasing rearrangement of the modulus of $u$, which lies in $S_r$. By standard properties $|\nabla v|_2 \le |\nabla u|_2$, $E(v) \le E(u)$, and $P(v) \le P(u) = 0$. If $P(v) = 0$ we have a contradiction, and hence we can assume that $P(v) <0$. In this case, from Lemma \ref{lem: fiber sup} we know that $t_v<0$. But then we obtain again a contradiction in the following way:
%\begin{align*}
%E(u) &< E(t_v \star v) =  \frac{\mu}q\left(\frac{\gamma_q q}2-1\right)|v|_q^q e^{\gamma_q q t_v} + \frac{1}{N}|v|_{2^*}^{2^*} e^{2^* t_v} \\
%& = \frac{\mu}q\left(\frac{\gamma_q q}2-1\right)|u|_q^q e^{\gamma_q q t_v} + \frac{1}{N}|u|_{2^*}^{2^*} e^{2^* t_v} \le e^{\gamma_q q t_v} E(u) < E(u),
%\end{align*}
%where we used the fact that $t_v \star v$ and $u$ lies in $\cP$. This proves that $m_r(a,\mu)=m(a,\mu)$, and hence $\tilde u$ is a ground state. 
\end{proof}

\begin{remark}\label{rem: on alpha}
As already mentioned in the introduction, assumption \eqref{hp} enters in the upper estimate of the ground state level $m(a,\mu)$, see Lemma \ref{lem: st above sub}, and in particular Step 2. We emphasize that, in \eqref{19114}, we could take advantage of the different integrability properties of $u_\eps$ in order to have $\alpha=+\infty$ for $N=3,4$.\end{remark}

\section{Non-existence in the defocusing case $\mu<0$}\label{sec: non ex}

\begin{proof}[Proof of Theorem \ref{thm: non-ex}] \textbf{1)} Let $u$ be a constrained critical point of $E_\mu$ on $S_a$. Then $u$ solves \eqref{stat com} for some $\lambda \in \R$; testing \eqref{stat com} by $\bar u$, the complex conjugate of $u$, and taking the real part, we deduce that
\[
|\nabla u|_2^2 = \lambda |u|_2^2 + \mu |u|_q^q + |u|_{2^*}^{2^*}.
\]
We also know that $P_\mu(u) = 0$, by the Pohozaev identity. Combining these identities, we infer that 
\[
\lambda |u|_2^2 = \mu (\gamma_q -1) |u|_q^q > 0,
\]
since $\mu<0$, $\gamma_q <1$, and $S_a \ni u \not \equiv 0$. Therefore $\lambda>0$. 

Using again the fact that $P_\mu(u) = 0$, by the Sobolev inequality we also deduce that
\[
|\nabla u|_2^2 \le |u|_{2^*}^{2^*} \le \cS^{-\frac{2^*}{2}} |\nabla u|_2^{2^*} \quad \implies \quad |\nabla u|_2^2 \ge \cS^\frac{N}2,
\]
whence, since $\mu<0$, 
\[
E_\mu(u) = \frac{1}{N} |\nabla u|_2^2  - \frac{\mu}q\left(1-\frac{\gamma_q q}{2^*}\right) |u|_q^q > \frac{\cS^\frac{N}2}{N},
\]
which completes the proof of point (1).

\medskip

\noindent \textbf{2 and 3)} Let $u$ be a solution to \eqref{non ex} for some $\lambda\in \R$, $\mu<0$. A Brezis-Kato argument \cite{BrKa} implies that $u$ is smooth, and is in $L^\infty(\R^N)$\footnote{Since we did not find a precise reference for the global boundedness of $u$, we included a proof in Appendix \ref{sec: BrKa}.}; thus, $|\Delta u| \in L^\infty(\R^N)$ as well, and standard gradient estimates for the Poisson equation (see formula (3.15) in \cite{GiTr}) imply that $|\nabla u| \in L^\infty(\R^N)$. This and the fact that $u \in L^2(\R^N)$ imply that $u(x) \to 0$ as $|x| \to \infty$. As a consequence, using the fact that $\lambda>0$ by point 1, we have that
\begin{equation}\label{u sup}
-\Delta u = (\lambda + \mu u^{q-2} + u^{2^*-2}) u \ge \frac{\lambda}{2} u >0 \quad \text{for $|x| > R_0$, with $R_0>0$ large enough};
\end{equation}
that is, $u$ is superharmonic at infinity. 
%By continuity, and since $u>0$, we have 
%\[
%m(R_0):= \inf_{|x|=R} u(x) >0.
%\]
%Since $u$ decays at infinity, there exists $R_1>R_0$ such that 
%\[
%\sup_{|x| \ge R_1} u < m(R_0).
%\]
%Let $\rho >R_1$. Since $u$ is superharmonic in $B_{R_0}^c$, the minimum of $u$ on $\overline{B_{\rho} \setminus B_{R_0}}$ is achieved on the boundary, and clearly cannot stay on $\pa B_{R_0}$ since 
%\[
%\sup_{|x| = \rho} u \le \sup_{|x| \ge R_1} u < m(R_0).
%\]
%That is, the minimum of $u$ on $\overline{B_{\rho} \setminus B_{R_0}}$ is achieved on $\pa B_\rho$, for every $\rho>R_1$. This implies that the function $m(r):= \inf_{|x| =r} u(x)$ is monotone decreasing in $r$ for $r>R_1$. 
By the Hadamard three spheres theorem \cite[Chapter 2]{ProWei}, this implies that the function $m(r):= \min_{|x|=r} u(x)$ satisfies
\[
m(r) \ge \frac{m(r_1)(r^{2-N}-r_2^{2-N}) + m(r_2) ( r_1^{2-N}- r^{2-N})      }{r_1^{2-N}- r_2^{2-N}} \qquad \forall R_0 < r_1 < r < r_2.
\]
Since $u$ decays at infinity, we have that $m(r_2) \to 0$ as $r_2 \to +\infty$, and hence taking the limit we infer that $r \mapsto r^{N-2} m(r)$ is monotone non-decreasing for $r>R_0$. Notice also that $m(r) >0$ for every $r>0$, since $u>0$ in $\R^N$. Therefore,
\[
m(r) \ge \underbrace{m(R_0) R_0^{N-2}}_{> 0} r^{2-N} \qquad \forall r>R_0,
\]
and this finally implies that
\begin{equation}\label{stima non ex}
\begin{split}
|u|_{p}^p & %= \int_{B_{R_0}} |u|^p + \int_{R_0}^{+\infty} \left( \int_{S_r} |u(r \theta)|^p \, d\theta\right)r^{N-1} \,dr
 \ge C \int_{R_0}^{+\infty} |m(r)|^p r^{N-1} dr  \ge C \int_{R_0}^{+\infty} r^{p(2-N) + N-1} dr,
\end{split}
\end{equation}
with $C>0$. If $N=3,4$, since $u \in H^1(\R^N)$ we have that $u \in L^{\frac{N}{N-2}}(\R^N)$. This choice of $p$ in \eqref{stima non ex} yields
\[
|u|_{\frac{N}{N-2}}^\frac{N}{N-2} \ge C \int_{R_0}^{+\infty} \frac{dr}{r} = +\infty,
\]
a contradiction. If instead $N \ge 5$, the fact that $u \in H^1(\R^N)$ does not imply that $u \in L^{\frac{N}{N-2}}(\R^N)$ or that $u \in L^p(\R^N)$ for some $p \in (0, N/(N-2)]$. But, imposing such condition as an assumption, we still reach a contradiction.
\end{proof}

\begin{remark}
The use of the three sphere theorem in the proof of non-existence results under suitable integrability condition is already present in the literature. We refer for instance to \cite[Lemma A.2]{Iko}.
\end{remark}

\section{Behavior of ground states when $\mu \to 0^+$}\label{sec: mu to 0}

In this section we prove Theorem \ref{thm: mu to 0}.

\begin{proof}[Proof of Theorem \ref{thm: mu to 0} for $2<q<2+4/N$]
We recall that $\tilde u_\mu$ is characterized as an interior local minimizer of $E_\mu$ on $\{u \in S_a: |\nabla u|_2<R_0\}$, where $R_0(a,\mu)$ is defined by Lemma \ref{lem: struct h}. Exactly as in \cite{So}, it is possible to check that $R_0=R_0(a,\mu) \to 0$ as $\mu \to 0^+$, then $|\nabla \tilde u_\mu|_2 < R_0(a,\mu) \to 0$ as well. Moreover, by the Gagliardo-Nirenberg and the Sobolev inequalities
\[
0 > m(a,\mu) = E_\mu(\tilde u_\mu) \ge \frac12 |\nabla \tilde u_\mu|_2 - \frac{\mu}{q} a^{(1-\gamma_q)q} C_{N,q}^q |\nabla \tilde u_\mu|_2^{\gamma_q q} - \frac{1}{2^* \cS^\frac{2^*}{2}} |\nabla \tilde u_\mu|_2^{2^*} \to 0
\]
as $\mu \to 0^+$.
\end{proof}

Let us consider now the more complicated case $2+4/N \le q < 2^*$. 
\begin{lemma}\label{inf sup}
Let $a>0$, and let $\mu \ge 0$ be such that \eqref{hp} is satisfied. Then
\[
\inf_{u \in \cP_{a,\mu}} E_\mu(u) = \inf_{u \in S_a} \max_{s \in \R} E_\mu(s \star u).
\]
\end{lemma} 
\begin{proof}
Since $2+4/N \le q<2^*$ and $\mu \ge 0$, we know from Lemmas \ref{lem: fiber cr} and \ref{lem: fiber sup} that $\cP_{a,\mu} = \cP_-^{a,\mu}$, that for every $u \in S_a$ there exists a unique $t_{u,\mu} \in \R$ such that $t_{u,\mu} \star u \in \cP_{a,\mu}$, and that $t_{u,\mu}$ is a strict maximum point for $\Psi_u^\mu$ (for the case $\mu=0$, see the proof of Proposition \ref{prop: hom}). Thus, if $u \in \cP_{a,\mu}$, we have that $t_{u,\mu}=0$, and
\[
E_\mu(u) =  \max_{s \in \R} E_\mu(s \star u) \ge \inf_{v \in S_a} \max_{s \in \R} E_\mu(s \star v).
\]
On the other hand, if $u \in S_a$, then $t_{u,\mu} \star u \in \cP_{a,\mu}$, and hence
\[
\max_{s \in \R} E_\mu(s \star u) = E_\mu(t_{u,\mu} \star u) \ge \inf_{v \in \cP_{a,\mu}} E_\mu(v). \qedhere
\]
\end{proof}

\begin{lemma}\label{mon m}
Let $a>0$, and let $\tilde \mu>0$ satisfy \eqref{hp} for this choice of $a$. Then the function $\mu \in [0, \tilde \mu] \mapsto m(a,\mu) \in \R$ is monotone non-increasing.
\end{lemma}

\begin{proof}
Let $0\le\mu_1\le \mu_2\le \tilde \mu$. By Lemma \ref{inf sup} we have that  
\[
\begin{split}
m(a,\mu_2) &= \inf_{u \in S_a} \max_{s \in \R} E_{\mu_2}(s \star u)  =\inf_{u \in S_a} E_{\mu_2}(t_{u,\mu_2} \star u) \\
& = \inf_{u \in S_a} \left[ E_{\mu_1}(t_{u,\mu_2} \star u) + \frac{\mu_1-\mu_2}{q}e^{\gamma_q q t_{u,\mu_2}} |u|_q^q \right] \le \inf_{u \in S_a} \max _{s \in \R} E_{\mu_1}(s \star u)  = m(a,\mu_1). \qedhere
\end{split}
\]
\end{proof}

\begin{proof}[Proof of Theorem \ref{thm: mu to 0} for $2+4/N \le q<2^*$]
Let $a>0$, and let $\tilde \mu>0$ satisfy \eqref{hp} for this choice of $a$. At first, we show that the family of positive radial ground states $\{\tilde u_\mu: \ 0<\mu< \tilde \mu\}$ is bounded in $H^1$. If $q= \bar p$, then by Lemma \ref{mon m}, and using the fact that $P_\mu(\tilde u_\mu)=0$, we have that
\[
m(a,0) \ge m(a,\mu) = E_\mu(\tilde u_\mu) = \frac1N|\nabla \tilde u_\mu|_2^2 - \frac{2\mu}{N \bar p} |\tilde u_\mu|_{\bar p}^{\bar p} \ge \frac1N\left(1-\frac{2\tilde \mu}{\bar p} C_{N,\bar p}^{\bar p} a^{\frac4N} \right) |\nabla \tilde u_\mu|_2^2.
\]
If $\bar p<q<2^*$, in a similar way we note that 
\[
m(a,0) \ge m(a,\mu) = E_\mu(\tilde u_\mu) = \frac1N| \tilde u_\mu|_{2^*}^{2^*} + \frac{\mu}{q}\left(\frac{\gamma_q q}{2}-1\right)|\tilde u_\mu|_q^q;
\]
since $\gamma_q q>2$, we deduce that $\{\tilde u_\mu\}$ is bounded in $L^q \cap L^{2^*}$, and hence, since $P_\mu(\tilde u_\mu) = 0$, it is also bounded in $H^1$. Since in particular $\{\tilde u_\mu\}$ is bounded in $L^q$, the associated sequence $\{\tilde \lambda_\mu\}$ of negative Lagrange multipliers tends to $0$:
\[
\tilde \lambda_\mu a^2 = |\nabla \tilde u_\mu|_2^2 - \mu |\tilde u_\mu|_q^q - |\tilde u_\mu|_{2^*}^{2^*} = \mu(\gamma_q-1)|\tilde u_\mu|_q^q \to 0 \quad \text{as $\mu \to 0^+$}.
\]
Therefore, we deduce that up to a subsequence $\tilde u_\mu \weak \tilde u$ weakly in $H^1$, in $\cD^{1,2}$, and in $L^{2^*}$; $\tilde u_\mu \to \tilde u$ strongly in $L^q$ and a.e. in $\R^N$; $|\nabla \tilde u_\mu|_2^2 \to \ell \ge 0$; $\tilde \lambda_\mu \to 0$. 

We claim that $\ell \neq 0$. Indeed, if $\ell = 0$, then $\tilde u_\mu \to 0$ strongly in $\cD^{1,2}(\R^N)$, and hence $E_\mu(\tilde u_\mu) \to 0$. However, by Lemma \ref{mon m} we know that $E_\mu(\tilde u_\mu) \ge m(a,\tilde \mu)>0$ for every $0<\mu<\tilde \mu$, a contradiction. 

Now, passing to the limit in the identity $P_\mu(\tilde u_\mu) = 0$, we deduce that 
\[
|\tilde u_\mu|_{2^*}^{2^*} = |\nabla \tilde u_\mu|_2^2 - \mu \gamma_q |\tilde u_\mu|_q^q \to \ell \quad \text{as $\mu \to 0^+$}
\]
as well. Therefore, by the Sobolev inequality $\ell \ge S \ell^\frac{2}{2^*}$, which implies that $\ell \ge \cS^\frac{N}2$, since $\ell \neq 0$. 

On the other hand, we also have
\[
\frac{\ell}{N} = \lim_{\mu \to 0^+} \left[\frac1N |\nabla \tilde u_\mu| - \frac{\mu}{q}\left(1-\frac{\gamma_q q}{2^*}\right)|\tilde u_\mu|_q^q\right] = \lim_{\mu \to 0^+} E_\mu(\tilde u) \le m(a,0) = \frac{\cS^\frac{N}2}{N},
\]
by Lemma \ref{mon m} and Proposition \ref{prop: hom}. This finally gives $\ell = \cS^\frac{N}2$, and the thesis follows.
\end{proof}

\begin{remark}\label{rem: on gs to 0}
It is natural to wonder if it is possible to characterize the asymptotic behavior of $\tilde u_\mu$, and not only of $|\nabla \tilde u_\mu|_2$. This is a delicate problem. Indeed, in the above proof, by weak convergence the weak limit $\tilde u$ solves
\begin{equation}\label{lim pb}
-\Delta \tilde u = \tilde u^{2^*-1}, \quad \tilde u \ge 0 \quad \text{in $\R^N$}, \quad
\tilde u \in H^1_{\rad}(\R^N),
\end{equation}
and we have two possible alternatives: $\tilde u \equiv 0$, or else $\tilde u \not \equiv 0$. The second alternative cannot hold if $N=3,4$, since the above problem has only the trivial solution in such cases. Therefore, in dimension $N=3,4$ we have that $|\nabla \tilde u_\mu|_2^2 \to \cS^\frac{N}2$, but $\tilde u_\mu \weak 0$ in $H^1$. In higher dimension, we think that it is natural to conjecture that $\tilde u$ is an element of the family $U_{\eps,0}$, defined in \eqref{def bubble}. 
\end{remark}

\section{Dynamical properties}\label{sec: additional}

In this section we briefly describe the proof of Theorems \ref{thm: inst}, \ref{thm: gwp}, and of Proposition \ref{prop: struct P}, which can be easily derived from similar results already proved in \cite{So}. We often omit the dependence on $a$ and $\mu$.

\begin{proof}[Conclusion of the proof of Proposition \ref{prop: struct P}]
In proving Theorem \ref{thm: main}, we have already shown that $\cP$ is a smooth manifold of codimension $1$ on $S$, and that point (1) and (2) in the thesis hold. It only remains to show that $\cP$ is a natural constraint. This can be done exactly as in \cite[Proposition 1.11]{So}.\end{proof}

\begin{proof}[Proof of Theorem \ref{thm: gwp}]
We can proceed by repeating word by word the proof of \cite[Theorem 1.13 - finite time blow-up]{So}. Note that the existence and uniqueness of a maximum point $t_{u,\mu}$ for the fiber map $\Psi_u^\mu$, which is needed in the proof, follows here from Lemmas \ref{lem: fiber sub}, \ref{lem: fiber cr}, and \ref{lem: fiber sup}. Moreover, we mention that the assumption $|x| u \in L^2(\R^N)$ is used in order to exploit the virial identity, see e.g. \cite[Lemma 6.1]{TaoVisZha}.
\end{proof}

\begin{proof}[Proof of Theorem \ref{thm: inst}]
In order to characterize the set of ground states as in \eqref{struct Z}, it is possible to argue as in the Sobolev subcritical cases: if $2<q<2+4/N$, we refer to \cite[Theorem 1.4]{So}, while for $2+4/N \le q < 2^*$ we refer to \cite[Theorem 1.7]{So}.

The strong instability of the ground states when $2+4/N \le q<2^*$ can be proved again as in \cite[Theorem 1.7]{So}, using Theorem \ref{thm: gwp}.
\end{proof}

\appendix

\section{Some useful estimates}

We collect several estimates regarding the asymptotic behavior of integrals involving the standard bubble. Similar estimates are contained e.g. in \cite{BreNir}. We recall the definition \eqref{def bub 2} of $U_\eps$, the bubble in $\R^N$ centered in the origin, with concentration parameter $\eps$. We use the notation
\begin{equation}\label{def K_i}
\begin{array}{l l}
K_{1} := |\nabla U_1|_2^2 %= \int_{\R^N} (2-N)^2 \frac{|x|^2}{(1+ |x|^2)^N}\,dx \\ 
& K_{3} := |U_1|_2^2 \qquad  \forall N \ge 5  \\%= \int_{\R^N}  \frac{1}{(1+ |x|^2)^N}\,dx \\
K_{2} := |U_1|_{2^*}^{2}& %=\int_{\R^N} \left( \frac{1}{1+|x|^2}\right)^{N-2}dx \qquad \forall N \ge 5 \\
K_{4} := |U_1|_q^q \qquad \forall q > \frac{N}{N-2}, \ \forall N \ge 3.
\end{array}
\end{equation}
Since $U_{\eps}$ is extremal for the Sobolev inequality, we have that  $K_1/K_2= \cS$.

Let also $\varphi \in C^\infty_c(\R^N)$ be a radial cut-off function with $\varphi \equiv 1$ in $B_1$, $\varphi \equiv 0$ in $B_2^c$, and $\varphi$ radially decreasing. 

\begin{lemma}\label{lem: app}
Denoting by $u_\eps:= \varphi U_\eps$, and by $\omega$ the area of the unit sphere in $\R^N$, we have:
\begin{align}
& |\nabla u_\eps|_2^2 = K_{1} + O(\eps^{N-2}) \label{grad u_eps} \\
& |u_\eps|_{2^*}^{2}  = \begin{cases} K_{2} + O(\eps^N) & \text{if $N \ge 4$} \\ K_2 + O(\eps^2) & \text{if $N=3$} \end{cases} \label{norm star u_eps} \\
&|u_\eps|_2^2  = \begin{cases} \eps^2 K_{3} + O(\eps^{N-2}) & \text{if $N \ge 5$} \\ 
\omega \eps^2 |\log \eps| + O(\eps^2) & \text{if $N=4$} \\
\omega \left( \int_0^2 \varphi(r)\,dr\right) \eps + O(\eps^2) & \text{if $N=3$}
\end{cases} \label{norm 2 u_eps} \\
& |u_\eps|_q^q = \eps^{N-\frac{N-2}2 q} (K_{4} + O(\eps^{(N-2)q - N})) \quad \begin{array}{l} \text{if $N \ge 4$ and $q \in (2,2^*)$,}\\ \text{and if $N = 3$ and $q \in (3,6)$.} \end{array} \label{norm q u_eps}
\end{align}
%
%
%
%\begin{cases} \eps^{N-\frac{N-2}2 q} (K_{4,N} + O(\eps^{(N-2)q - N})) & \text{if $N \ge 4$, or if $N \ge 3$ and $q \in (3,6)$} \\ 
%K_{5,N}^q \eps^{\frac{3}2} (|\log \eps| + O(1)) & \text{if $N=3$ and $q=3$}  \\
%K_{6,N}^q \eps^{\frac{q}2}(1+o(1)) & \text{if $N=3$ and $q \in (2,3)$}\end{cases} \label{norm q u_eps}
%\end{align}
as $\eps \to 0$.
\end{lemma}

\begin{proof}
We observe that
\[
u_\eps(x) = \eps^\frac{N-2}2 \frac{\varphi(x)}{(\eps^2 + |x|^2)^\frac{N-2}2}.
\]
Therefore, the first three estimates follow as in (1.11)-(1.13) and (1.27)-(1.29) in \cite{BreNir}. Concerning the $L^q$ norm of $u_\eps$, we observe that if either $N \ge 4$ and $q  \in (2,2^*)$, or $N=3$ and $q \in (3,6)$, we have $q>N/(N-2)$. Hence $U_\eps \in L^q(\R^N)$, and more precisely
\begin{align*}
|u_\eps|_q^q & = \int_{\R^N} \varphi^q(x) \left(\frac{\eps}{\eps^2 + |x|^2}\right)^{\frac{N-2}2q}dx = \eps^{N- \frac{N-2}2 q} \int_{\R^N} \varphi^q(\eps y) \left( \frac1{1+|y|^2}\right)^{\frac{N-2}2 q}dy \\
%& = \eps^{N- \frac{N-2}2 q} \left[ \int_{\R^N} \left( \frac1{1+|y|^2}\right)^{\frac{N-2}2 q}\,dy +  \int_{\R^N} (\varphi^q(\eps y) -1) \left( \frac1{1+|y|^2}\right)^{\frac{N-2}2 q}dy \right] \\
& = \eps^{N- \frac{N-2}2 q}\left[ K_{4} + \int_{\R^N} (\varphi^q(\eps y) -1) \left( \frac1{1+|y|^2}\right)^{\frac{N-2}2 q}dy\right],
\end{align*}
and the desired estimate follows from the fact that
\begin{align*}
\bigg|\int_{\R^N} (\varphi^q(\eps y) -1) & \left( \frac1{1+|y|^2}\right)^{\frac{N-2}2 q}dy\bigg| \le \int_{\R^N \setminus B_{1/\eps}}  \left( \frac1{1+|y|^2}\right)^{\frac{N-2}2 q}dy \\
& \le \omega \int_{1/\eps}^{+\infty} r^{N-1 - (N-2) q}\,dr = \frac{\omega}{(N-2)q-N} \eps^{(N-2)q-N}. \qedhere
\end{align*}
\end{proof}

\section{Boundedness of solutions}\label{sec: BrKa}

In this appendix we present a proof of the following fact, which we think is known, but for which we could not find a proper reference.

\begin{proposition}
Let $f:\R^{N+1} \to \R$ be such that
\begin{equation}\label{hp su f}
|f(x,s)| \le C_0(|s|^{p-1} + |s|^{2^*-1}), \quad \text{for some} \quad p \in [2,2^*], \ C_0 >0.
\end{equation}
If $u \in \mathcal{D}^{1,2}(\R^N) \cap L^p(\R^N)$ is a real valued weak solution to 
\begin{equation}\label{eq: BK}
-\Delta u = f(x,u) \qquad \text{in $\R^N$},
\end{equation}
then $u \in L^\infty(\R^N)$.
\end{proposition}

The following proof is contained in some unpublished notes of Miguel Ramos, and was communicated to us by Hugo Tavares, to which we express our gratitude.

\begin{proof}
By \eqref{hp su f}, we have that
\begin{equation}\label{def a}
-\Delta u = a(x) u \quad \text{in $\R^N$, where} \quad |a(x)| = \frac{|f(x,u(x))|}{|u(x)|} \le C_0 ( |u(x)|^{p-2}+ |u(x)|^{2^*-2}).
\end{equation}
Thus, if $x \in \{|a(x)|> C_1\}$ with $C_1>2 C_0$, we have that $|u(x)| >1$, and $|a(x)| \le 2 C_0 |u(x)|^{2^*-2}$. This shows that $|a| \chi_{\{ |a| > C_1\}} \in L^{N/2}(\R^N)$, and implies that
\[
\eps(K) := \left( \int_{\{|a| > K\}} |a|^\frac{N}2\right)^\frac{2}{N} \to 0 \quad \text{as $K \to \infty$}.
\]
Using this fact, it is not difficult to use a classical Brezis-Kato argument as in \cite{BrKa} (or \cite[Appendix B]{Str}) and show that $u \in L^q(\R^N)$ for every $q \in [p,+\infty)$, and moreover that $|u|^{s+1} \in H^1(\R^N)$ for every $s \ge (p-2)/2$. For any such $s$, let $w:= |u|^{s+1}$. We observe that
\[
\nabla u \cdot \nabla (|u|^{2s} u) = \frac{2s+1}{(s+1)^2} |\nabla |u|^{s+1}|^2 \ge \frac{1}{s+1} |\nabla |u|^{s+1}|^2 = \frac{1}{s+1} |\nabla w|^2.
\]
Hence, testing the equation of $u$ with $|u|^{2s} u$, we obtain
\[
\int_{\R^N} |\nabla w|^2 \le \alpha \int_{\R^N} |a| w^2,
\]
where we set $\alpha:= s+1$ in order to simplify the notation; note that, by \eqref{def a},
\[
a \in L^q(\{|u|>1\}) \qquad \text{for every }q \in \left[\frac{p}{2^*-2},+\infty\right).
\]
Let then $q > \max\{2p/(2^*-2), N\}$, in such a way that 
\[
C_2:= \left| |a| \chi_{\{|u|>1\}}\right|_{\frac{q}{2}} <+\infty, \quad \text{and} \quad 2^* > \frac{2q}{q-2} =: \bar q.
\]
By the H\"older and the Sobolev inequalities
\[
\begin{split}
|w|_{2^*}^2 &\le C \int_{\R^N} |\nabla w|^2 \le C\alpha \int_{\{|u| \le 1\}} |a| w^2 + C\alpha \int_{\{|u| > 1\}} |a| w^2 \\
& \le C\alpha \int_{\{|u| \le 1\}} |u|^p + C_2 \alpha |w|_{\bar q}^2 \le C\alpha  + C_2 \alpha |w|_{\bar q}^2.
\end{split}
\]
Since $w=|u|^{s+1}= |u|^\alpha$, we infer that
\begin{equation}\label{ap B 1}
|u|_{\alpha \bar q \frac{2^*}{\bar q}} = |u|_{\alpha 2^*} \le (C \alpha)^\frac{1}{\alpha} \left( 1+ |u|_{\alpha \bar q}\right).
\end{equation}
Using this estimate, we can finally show that $|u|_r$ is uniformly bounded for a sequence of diverging exponents $r_n \to +\infty$. If $|u|_r \le 1$ for a sequence $r \to \infty$, then there is nothing to prove, therefore we can suppose that $|u|_r >1$ for every $r$ large enough. Thus, estimate \eqref{ap B 1} gives
\begin{equation}\label{ap B 2}
|u|_{\alpha \bar q \frac{2^*}{\bar q}}  \le (C \alpha)^\frac{1}{\alpha}  |u|_{\alpha \bar q}
\end{equation}
for every $\alpha=s+1$ with $s \ge (p-2)/2$. In particular, taking $\alpha = (2^*/\bar q)^{i+k}$ with $i \in \N$ and $k$ large enough, and iterating \eqref{ap B 2}, we deduce that
\begin{equation}
\begin{split}
|u|_{\bar q \left(\frac{2^*}{\bar q}\right)^{k + n}} &\le \prod_{i=0}^{n-1} \left( C \left(\frac{2^*}{\bar q}\right)^{k+i}\right)^{\left(\frac{2^*}{\bar q}\right)^{-(k+i)}} |u|_{ \bar q\left(\frac{2^*}{\bar q}\right)^k} \\
& = \exp \left\{ \sum_{i=0}^{n-1} \left(\frac{2^*}{\bar q}\right)^{-(k+i)} \log \left( C \left(\frac{2^*}{\bar q}\right)^{k+i}\right) \right\} |u|_{\bar q\left(\frac{2^*}{\bar q}\right)^k }.
\end{split}
\end{equation} 
Taking the limit as $n \to +\infty$, and using the converge of the sum inside the brackets (since $2^*> \bar q$), we deduce that 
\[
\lim_{n \to \infty} |u|_{r_n} \le C |u|_{\left(\frac{2^*}{\bar q}\right)^k \bar q} =:\bar C, \quad \text{for} \quad  r_n:= \bar q \left(\frac{2^*}{\bar q}\right)^{k + n} \to +\infty. 
\]
It is not difficult at this point to infer that $u \in L^\infty(\R^N)$: indeed, for any $K> \bar C$, we have that
\[
 |\{|u|>K\}|^\frac{1}{r_n} K \le |u|_{r_n} \le \bar C \quad \implies \quad  |\{|u|>K\}| \le \left( \frac{\bar C}{K}\right)^{r_n} \to 0
 \]
as $n \to \infty$, since $\bar C<K$ and $r_n \to \infty$.
\end{proof}

%
%\bibliography{Soa}
%\bibliographystyle{abbrv}

\end{document}